\documentclass[a4paper, 11pt]{article}

\usepackage{a4wide, fancyhdr, amsmath, amssymb, mathtools, yfonts}
\usepackage{mathrsfs}
\usepackage{graphicx}
\usepackage{tikz}
\usepackage[all]{xy}
\usepackage[utf8]{inputenc}
\usepackage{amsthm}
\usepackage[english]{babel}
\usepackage{chngcntr}
\usepackage{ifthen}
\usepackage{calc}
\usepackage{hyperref}
\usepackage{authblk}
\usepackage{tikz-cd}
\numberwithin{equation}{section}
\newcommand{\Z}{\mathbb{Z}}
\newcommand{\Q}{\mathbb{Q}}

\newcommand\Gal{\mathrm{Gal}}

\newtheorem{lemma}{Lemma}[section]
\newtheorem{theorem}[lemma]{Theorem}
\newtheorem{proposition}[lemma]{Proposition}
\newtheorem{corollary}[lemma]{Corollary}
\newtheorem{definition}[lemma]{Definition}
\newtheorem{conjecture}[lemma]{Conjecture}
\newtheorem{remark}[lemma]{Remark}

\title{\vspace{-\baselineskip}\sffamily\bfseries Two-step nilpotent extensions are not anabelian}
\author[1]{Peter Koymans\thanks{Department of Mathematics, Ann Arbor, MI 48109, USA, koymans@umich.edu}}
\author[2]{Carlo Pagano\thanks{Department of Mathematics and Statistics, Montreal, Quebec H3G 1M8, Canada, carlein90@gmail.com}}
\affil[1]{University of Michigan}
\affil[2]{Concordia University}
\date{\today}

\begin{document}
\maketitle

\begin{abstract}
We prove the existence of two non-isomorphic number fields $K$ and $L$ such that the maximal two-step nilpotent quotients of their absolute Galois groups are isomorphic. In particular, one may take $K$ and $L$ to be any of the fields $\Q(\sqrt{-11})$, $\Q(\sqrt{-19})$, $\Q(\sqrt{-43})$, $\Q(\sqrt{-67})$ or $\Q(\sqrt{-163})$. Furthermore, we give an explicit combinatorial description of these Galois groups in terms of a generalization of the Rado graph. A critical ingredient in our proofs is the back-and-forth method from model theory.
\end{abstract}

\section{Introduction}
\subsection{History of anabelian geometry}
The anabelian program is a broad set of conjectures aimed at reconstructing arithmetic and geometry from Galois theory. An early gem in this direction, established before the program was put in motion, is the celebrated Neukirch--Uchida theorem from the late 1960s, which shows that one may functorially recover number fields from their absolute Galois groups. In particular, one has that 
\begin{align}
\label{eAnabelianNeu}
\mathcal{G}_K \simeq_{\text{top.gr.}} \mathcal{G}_L \Longleftrightarrow K \simeq_{\text{fields}} L,
\end{align}
where $\mathcal{G}_K$, $\mathcal{G}_L$ denote the absolute Galois groups of the respective number fields.

The starting point of the proof was Neukirch's insight that there is a $p$-adic counterpart to Artin--Schreier's characterization of real closed fields. This insight culminated into a topological characterization of the decomposition groups, which is often referred to as the \emph{local theory} of Neukirch's argument. Already the local theory combined with Chebotarev's density theorem readily implies equation (\ref{eAnabelianNeu}) for normal extensions. It took the efforts of Neukirch, Ikheda, Iwasawa and finally Uchida \cite{Uchida} to arrive at a substantial functorial strengthening of equation (\ref{eAnabelianNeu}).

The term \emph{anabelian}, which stands for \emph{beyond abelian}, goes back to Grothendieck's letter \cite{Grothendieck} to Faltings in $1983$, where a vast conjectural program was set forth. He predicted that there should be an analogue of the Neukirch--Uchida theorem for finitely generated fields, which was established successively by Pop \cite{Pop1} and Mochizuki \cite{Moch: fin.generated}. Another influential prediction was a vast generalization of Mordell's conjecture, known as the section conjecture, where, under suitable conditions, rational points on varieties are identified precisely as the set of group-theoretic sections of the map $\pi_1(X) \twoheadrightarrow \mathcal{G}_K$. Remarkable progress on the section conjecture can be found in the work of Koenigsmann \cite{Koen2}, Stix \cite{Stix,Stix2,Stix3}, the work of Pop--Stix \cite{PopStix}, Pop \cite{Pop6} and Betts--Stix \cite{Betts-Stix}.  

\subsection{Main results}
In the early 1990s, Bogomolov \cite{Bo1} set forth an anabelian program for fields of \emph{geometric nature}. In particular, the goal is to reconstruct the function field $k(X)$ (up to an inseparable extension) of a variety $X$ of dimension at least $2$ over an algebraically closed field $k$ from its absolute Galois group $\mathcal{G}_{k(X)}$. Even more, the program predicts that one may reconstruct $k(X)$ from $\mathcal{G}_{k(X)}^2(\ell)$, the largest $2$-nilpotent, pro-$\ell$ quotient of $\mathcal{G}_{k(X)}$, where $\ell$ is a prime assumed to be different from the characteristic $\text{char}(k)$ of the base field $k$. 

An analogue of Neukirch's local theory on the geometric side can be traced back to a fundamental insight of Bogomolov in \cite{Bo1} and expanded in \cite{Bo2}. This circle of ideas has been vastly expanded by Pop \cite{Pop1, Pop2, Pop3, Pop4}. As a result of his work, Bogomolov's program is solved when $k$ is the algebraic closure of a finite field, or more generally when $\text{tr-deg}(k(X))>\text{tr-deg}(k)+1$. The restriction $\text{dim}(X) \geq 2$ is necessary, since $\mathcal{G}_{k(X)}(\ell)$ is free pro-$\ell$ on $\#k$ generators if $X$ is a curve. We refer to Pop's survey article \cite{Pop5} for an overview of anabelian geometry both in the arithmetic and geometric case. For an analogue of the section conjecture in this setting, we refer the reader to the work of Bogomolov--Rovinsky--Tschinkel \cite{Bo4}. 

Given that in the birational program described above the reconstruction takes place from much smaller quotients of the absolute Galois group, it is natural to ask which quotients of the absolute Galois groups of number fields are rigid enough to encode the arithmetic of the field: whether the ones appearing in Bogomolov's program suffice is precisely the question we are able to settle here. 

More formally, let $K^2$ be the compositum, inside $K^{\text{sep}}$, of all the finite Galois extensions of $K$ such that their Galois group is nilpotent of class at most $2$. We write $\mathcal{G}_K^2 = \Gal(K^2/K)$ for the corresponding Galois group.

\begin{theorem} 
\label{Thm: main1}
There exist two non-isomorphic number fields $K_1$ and $K_2$ such that 
$$
\mathcal{G}_{K_1}^2 \simeq_{\textup{top.gr.}} \mathcal{G}_{K_2}^2.
$$
\end{theorem}

A discussion of the substantial literature on Galois theoretic reconstruction results takes place in the next subsection. Our main theorem is a consequence of the following more precise result.  We write $\mathcal{G}(\ell)$ for the maximal pro-$\ell$ quotient of a profinite group $\mathcal{G}$. 

\begin{theorem} 
\label{Thm: main2}
Let $\ell$ be a prime number. Let $K_1$, $K_2$ be two imaginary quadratic number fields such that $\textup{Cl}(K_i)[\ell] = \{0\}$ for each $i \in \{1, 2\}$. For $\ell = 2$, assume furthermore that $2$ is inert in both $K_1$ ,$K_2$ and, for $\ell = 3$, assume that $K_i \otimes_{\Q} \Q_3$ is not isomorphic to $\Q_3(\zeta_3)$ as $\Q_3$-algebra. Then
$$
\mathcal{G}_{K_1}^2(\ell) \simeq_{\textup{top.gr.}} \mathcal{G}_{K_2}^2(\ell).
$$
\end{theorem}

Observe that for each $\ell$, this theorem applies to infinitely many quadratic fields thanks to a result of Wiles \cite{Wiles}. In particular, for $\ell = 2$, Gauss' genus theory and Theorem \ref{Thm: main2} imply that all $\mathcal{G}_{\Q(\sqrt{-p})}^2(2)$ are isomorphic to an explicit group as $p$ runs through the positive rational primes congruent to $3$ modulo $8$. As we shall see later, we are able to explicitly describe this group, which we coin the \emph{Rado group}, in Section \ref{Rado groups at 2}. From Theorem \ref{Thm: main2}, we obtain the following corollary, which is an explicit version of Theorem \ref{Thm: main1}. 

\begin{corollary} 
\label{Cor: main3}
Let $K$, $L$ be any of the fields $\Q(\sqrt{-11})$, $\Q(\sqrt{-19})$, $\Q(\sqrt{-43})$, $\Q(\sqrt{-67})$ or $\Q(\sqrt{-163})$. Then
$$
\mathcal{G}_K^2 \simeq_{\textup{top.gr.}} \mathcal{G}_L^2.
$$
\end{corollary}

\subsection{Discussion of main results}
There is a rich literature on similar Galois theoretic reconstruction problems, although these results are almost exclusively successful in reconstructing the number field from Galois theoretic data. Theorem \ref{Thm: main1} therefore takes an interesting place in the literature as the first negative Galois theoretic reconstruction result since the results of Onabe \cite{Onabe} from the 1970s. The theory of arithmetic equivalence as initiated by Perlis \cite{Perlis} also provides ample examples of negative reconstruction results, but these are not of Galois theoretic nature. We will now discuss these Galois theoretic reconstruction questions and their relation with our main results.

It is not hard to adapt the proof of Neukirch--Uchida to establish that 
$$
\mathcal{G}_K^{\text{pro-solv}} \simeq_{\text{top.gr.}} \mathcal{G}_L^{\text{pro-solv}} \Longleftrightarrow K \simeq_{\text{fields}} L.
$$
Only very recently a substantial leap was made by Saidi and Tamagawa \cite{Saidi-Tama} . They arrived in particular at the striking conclusion that one may reconstruct $K$ from $\mathcal{G}_K^{3\text{-solv}}$, the largest $3$-solvable quotient. Namely, they proved that
$$
\mathcal{G}_K^{3\text{-solv}} \simeq_{\text{top.gr.}} \mathcal{G}_L^{3\text{-solv}} \Longleftrightarrow K \simeq_{\text{fields}} L.
$$
The case of $2$-solvable quotients remains open. Although our main theorem does not settle this open question, it may still provide important insight as $\mathcal{G}_K^2$ and $\mathcal{G}_K^{2\text{-solv}}$ are rather closely related. Saidi--Tamagawa was inspired by earlier work of Cornelissen--de Smit--Li--Marcolli--Smit \cite{Co1} and Cornelissen--Li--Marcolli--Smit \cite{Co2}, which reconstructed $K$ from $\mathcal{G}_K^1$ together with extra data on the Dirichlet $L$-functions of $K$. Extensive studies of $\mathcal{G}_K^1$ were conducted by Angelakis--Stevenhagen \cite{St} and de Smit--Solomatin \cite{de Smit, de Smit 2}, building upon previous work of Kubota \cite{Kubota} and Onabe \cite{Onabe}. Examples of quadratic number fields $K$ with isomorphic $\mathcal{G}_K^{\text{ab}} = \mathcal{G}_K^{1\text{-solv}}$ feature prominently in several of the above works.

A local analogue of the Neukirch--Uchida theorem was established by Mochizuki \cite{Moch:local}. Here, as shown by Jarden--Ritter \cite{Jarden}, the naive analogue of the theorem \emph{fails}, and one has to consider not merely isomorphisms of topological groups, but rather isomorphisms of \emph{filtered} topological groups, where the filtration on both sides is the upper-ramification filtration. We also mention the work of Ivanov \cite{Ivanov} and Shimizu \cite{Shimizu} for an intriguing extension of Neukirch--Uchida to groups with restricted ramification.  

\subsection{Method of proof}
Theorem \ref{Thm: main2} relies on a concrete combinatorial description of the groups $\mathcal{G}_K^{2}(\ell)$ in terms of \emph{decorated Rado graphs}, a generalization of the classical notion of a Rado graph, discovered independently by Ackermann \cite{Ackermann} and Rado \cite{Rado}. This graph is also the central player in seminal work by Erd\H{o}s--R\'enyi \cite{Erdos--Renyi}, where the Rado graph was shown to be isomorphic to a random graph on countably many vertices with probability $1$. We achieve our combinatorial description by an application of the \emph{back-and-forth} method from model theory (see for example \cite{Poizat}) in the following manner. Thanks to a result of Tate, one has that $H^2(\mathcal{G}_K, \Q/\Z) = 0$ for any number field $K$. As observed already in \cite{MO}, this conveniently gives a natural identification
$$
\wedge_{\text{prof}}^2(\mathcal{G}_K^{\text{ab}}) = [\mathcal{G}_K^{2}, \mathcal{G}_K^{2}].
$$
Hence the group $\mathcal{G}_K^{2}(\ell)$ sits in the exact sequence
$$
0 \to \wedge_{\text{prof}}^2(\mathcal{G}_K^{\text{ab}}(\ell)) \to \mathcal{G}_K^{2}(\ell) \to \mathcal{G}_K^{\text{ab}}(\ell) \to 0,
$$
where the groups $\mathcal{G}_K^{\text{ab}}(\ell)$, $\wedge_{\text{prof}}^2(\mathcal{G}_K^{\text{ab}}(\ell))$ are both explicitly given (see for instance \cite{St, Onabe} for the former group and see Subsection \ref{intermezzo} for background on the profinite exterior square of a profinite abelian group). Then, in order to encode the group $\mathcal{G}_K^{2}(\ell)$, one has to describe the \emph{extension class} between these two explicit groups. 

It is at this point where our fundamental new insight comes in by connecting this extension class to the Rado graph. We show that this extension class can be encoded in a \emph{decorated graph}, which are formally introduced in Section \ref{decorated odd graphs} for $\ell$ odd. In order to achieve this, we happen to need a strengthening of Tate's result that $H^2(\mathcal{G}_K, \Q/\Z) = 0$, which is abstracted in part (C3), (C4) of Definition \ref{def: Rado groups} for $\ell$ odd.

The key next step is then showing that the graphs obtained this way are all \emph{pseudo-random} in the following sense. They share with $100\%$ of decorated graphs an \emph{extension property} that characterizes their isomorphism classes: it is precisely here that we apply the back-and-forth method. The desired pseudo-randomness is ultimately a consequence of the Chebotarev density theorem. We remark that these graphs are closely related to R\'edei matrices, whose randomness properties have played a crucial role in several recent developments in arithmetic statistics, see for instance \cite{KP} for an overview. 

It is instructive to compare the situation with the one of Bogomolov's program mentioned above. Unlike global fields, and unlike geometric function fields of curves, the second cohomology group $H^2(\mathcal{G}_{k(X)}, \Q/\Z)$ is far from trivial when $\text{dim}(X) \geq 2$. Hence in contrast to the case just explained, the commutator subgroup $[\mathcal{G}_{k(X)}^2(\ell), \mathcal{G}_{k(X)}^2(\ell)]$ is not explicitly described in terms of the abelianization. In other words, the natural surjection
$$
\wedge_{\text{prof}}^2(\mathcal{G}_{k(X)}^{\text{ab}}(\ell)) \twoheadrightarrow [\mathcal{G}_{k(X)}^2(\ell), \mathcal{G}_{k(X)}^2(\ell)]
$$
is not an isomorphism. The kernel of this map is in turn precisely what is exploited in Bogomolov's \emph{theory of commuting pairs} built by Bogomolov--Tschinkel \cite{ Bo2, Bo3} following an original idea of Bogomolov \cite{Bo1}. The theory of commuting pairs is the starting point for an analogue of Neukirch's local theory. The first works in this direction were due to Ware \cite{Ware} for $\ell = 2$ and Koenigsmann \cite{Koenigsman} for $\ell$ odd. A complete treatment of the theory was subsequently given by Topaz \cite{Topaz}. For an overview of Bogomolov's program we refer to Pop \cite[Historical note]{Pop4}.

In contrast, the commutator pairing is entirely predictable in our situation. It is in this respect that $\mathcal{G}_K(\ell)$ behaves similarly to a free pro-$\ell$ group, which is, as we mentioned above, what one gets for a curve $X$ over an algebraically closed $k$: the $1$-dimensional nature of $\text{Spec}(\mathcal{O}_K)$ here is manifested in the vanishing of $H^2(\mathcal{G}_K, \Q/\Z)$, which is shared with the case of curves over algebraically closed fields. It is only in the latter geometric case that one has the stronger vanishing $H^2(\mathcal{G}_K, \mathbb{F}_\ell) = 0$, forcing in this case the groups $\mathcal{G}_K(\ell)$ to be actually free pro-$\ell$. This larger class of ``free-like" pro-$\ell$ groups is what we will define as \emph{free on their abelianization} right before Proposition \ref{Enough surjective on abelianization}. 

That said, the arithmetic case of global fields manifests a fundamental difference with the geometric $1$-dimensional case, as the groups $\mathcal{G}_K(\ell)$ in the arithmetic case are actually \emph{never free} (the abelianization has torsion) and this creates the problem of describing the isomorphism class of the extension class between the two known groups. The novel achievement of the present work is precisely the description of this class in terms of \emph{decorated Rado graphs}. This partial resemblance with free pro-$\ell$ groups in dimension $1$, responsible for our combinatorial characterization, resonates with work of Hoshi disproving the pro-$\ell$ section conjecture \cite{Hoshi}. A key step in Hoshi's proof is the fact that the Galois group $\mathcal{G}_{\Q(\zeta_\ell)}^{\text{unr},\ell}(\ell)$ is a free pro-$\ell$ group when $\ell$ is a regular prime, which can be shown using techniques dating back to Shafarevich \cite{Sha}. 

\subsection{Future directions}
Consistently with this picture, our proof proceeds by shuffling places. Inspired by \cite{Saidi-Tama} one could ask whether, modulo twists by center-valued characters, this is the only possible isomorphism. Thanks to our Rado graph description, this is purely a problem in infinitary combinatorial linear algebra, which we have been unable to answer. However, in the positive direction, we show that one may already reconstruct the cyclotomic character (in the following precise sense) from the topological group $\mathcal{G}_K^{2}(\ell)$ for the fields as in Theorem \ref{Thm: main2}, a result that is interesting to compare to \cite[Appendix 4]{Saidi-Tama}. 

\begin{theorem} 
\label{Thm: reconstruction}
Let $\ell$ be a prime. Let $K$ be imaginary quadratic with $\textup{Cl}(K)[\ell] = \{0\}$.

\begin{enumerate}
\item[(a)] Suppose that $\ell$ is odd. If $\ell = 3$, we further require that $K \otimes_{\Q} \Q_3$ is not isomorphic to $\Q_3(\zeta_3)$. Then the kernel of the $\ell$-adic cyclotomic character
$$
\textup{ker}(\pi_{\textup{cyc}}(\ell): \mathcal{G}_K^{2}(\ell) \twoheadrightarrow 1 + \ell \cdot \Z_\ell)
$$
is a characteristic subgroup of $\mathcal{G}_K^{2}(\ell)$, i.e. it is preserved by $\textup{Aut}_{\textup{top.gr.}}(\mathcal{G}_K^{2}(\ell))$. 
\item[(b)] Suppose that $2$ is inert in $K$, i.e. $K := \Q(\sqrt{-p})$ for $p \equiv 3 \bmod 8$ a prime. Then both the kernel of the $2$-adic cyclotomic character 
$$
\textup{ker}(\pi_{\textup{cyc}}(2): \mathcal{G}_K^{2}(2) \to 1 + 4 \cdot \Z_2)
$$
and the kernel of $\chi_{-1}$
$$
\textup{ker}(\chi_{-1}: \mathcal{G}_K^{2}(2) \twoheadrightarrow \mathbb{F}_2)
$$
are characteristic subgroups of $\mathcal{G}_K^{2}(2)$, i.e. they are preserved by $\textup{Aut}_{\textup{top.gr.}}(\mathcal{G}_K^{2}(2))$.
\end{enumerate}
\end{theorem}
We actually prove a stronger result, providing characterizing properties for these characters, see the proof of Proposition \ref{reconstructing cyclotomic abstract} and the statement of Theorem \ref{chi-1 in general}. This reconstruction result has the following interesting consequence.

\begin{corollary}
\label{cQ7}
Let $K \in \{\Q(\sqrt{-11}), \Q(\sqrt{-19}), \Q(\sqrt{-43}), \Q(\sqrt{-67}), \Q(\sqrt{-163})\}$. Then we have that $\mathcal{G}_{\Q(\sqrt{-7})}^2 \not \simeq_{\textup{top.gr.}} \mathcal{G}_K^2$ .
\end{corollary}

The corollary demonstrates that there is additional information contained in $2$-step nilpotent Galois groups beyond the abelianization, since it is well-known \cite{St, Onabe} that the abelianized Galois groups are in fact isomorphic for the fields considered in Corollary \ref{cQ7}. It also shows that, even in the case $\text{Cl}(K) = \{0\}$, we can not remove the condition that $2$ is inert in $K$ in Theorem \ref{Thm: main2}.

Theorem \ref{Thm: main2} and Corollary \ref{Cor: main3} beg two questions: \\
\textbf{Question 1:} Can one explicitly describe $\mathcal{G}_K^{2}(\ell)$ for all global fields $K$ and all primes $\ell$? \\
\textbf{Question 2:} Can one deal with arbitrarily large nilpotency class?

We spend a few final words on these two questions and other future directions. 

Regarding question 1, already the description of $\mathcal{G}_K^{\text{ab}}$ becomes rather more involved: this was done by Kubota \cite{Kubota} in the language of Ulm invariants, which are explicitly given in terms of the arithmetic of the base field. It would be very interesting to find an appropriate extension of Theorem \ref{Thm: main2} to all number fields, with a suitable enlargement of our category of Rado groups from Section \ref{Rado group odd primes} and Section \ref{Rado groups at 2}. Towards this goal, we remark that our techniques may readily be adapted to pin down the isomorphism type of $\mathcal{G}_\Q^2$ in terms of combinatorial data, which in itself is already a novel and exciting achievement. Earlier work of Berrevoets \cite{Berrevoets} gave a description of $\mathcal{G}_\Q^2$ of a more algebraic nature.

Regarding question 2, we believe that it is possible to encode all the successive quotients of the lower central series in suitably \emph{decorated hypergraphs} for the fields of Corollary \ref{Cor: main3} and for $K := \Q$ (and more generally, for all the fields of Theorem \ref{Thm: main2} if $\ell$ is fixed). We also believe that the resulting hypergraphs are pseudo-random, in the sense that they enjoy a characterizing extension property that is also enjoyed by $100\%$ of the decorated hypergraphs. Our intuition is that one may prove this $100\%$ result by utilizing a suitable adaptation of our back-and-forth method.

However, to the best of our understanding, it is extremely hard to \emph{prove} that the hypergraphs arising from Galois groups are indeed pseudo-random. In the present work, this ultimately falls as a consqequence of the Chebotarev density theorem. Already for $3$-nilpotent quotients, Chebotarev will only control \emph{some} of the labels of the hypergraph but not all of them. In particular, it seems that if one would be able to carry out the same strategy in higher nilpotency class, then would get as a by-product a positive answer to the \emph{minimal ramification problem}, which is still open for nilpotent groups. Previous work by Kisilevski--Sonn \cite{Kis1} and Kisilevski--Neftin--Sonn \cite{Kis2} covers certain classes of nilpotent groups. In that vein, we also mention the work of Shusterman \cite{Shusterman}, who made progress over function fields in the large $q$ limit for general groups. 

That being said, the minimal ramification problem is however settled in nilpotency class $2$ due to work of Plans \cite{Plans} for $\ell$ odd and work of Kisilevski--Sonn \cite{Kis1} for general $\ell$: this can be readily reproved using the techniques of this paper, as the minimal ramification problem has a combinatorial counterpart for our Rado groups, which is not difficult to establish. As this result is already known, we leave the details of this approach to the interested reader. 

For a profinite group $\mathcal{G}$, we denote by $\mathcal{G}^{\text{pro-nil}}$ the inverse limit of all finite nilpotent quotients of $\mathcal{G}$. Altogether, this leads us to make the following conjectures. 

\begin{conjecture}
\label{cNil}
Let $K$, $L$ be any of the fields $\Q(\sqrt{-11})$, $\Q(\sqrt{-19})$, $\Q(\sqrt{-43})$, $\Q(\sqrt{-67})$ or $\Q(\sqrt{-163})$. Then
$$
\mathcal{G}_{K}^{\textup{pro-nil}} \simeq_{\textup{top.gr.}} \mathcal{G}_{L}^{\textup{pro-nil}}.
$$
\end{conjecture}

\begin{conjecture}
Let $\ell$ be a prime and let $K$, $L$ be any fields as in Theorem \ref{Thm: main2}. Then
$$
\mathcal{G}_{K}(\ell) \simeq_{\textup{top.gr.}} \mathcal{G}_{L}(\ell).
$$
\end{conjecture}

In fact, we expect that it should be possible to substantially refine the above two conjectures. Such conjectures would involve a pro-nilpotent analogue of the Rado group defined using \emph{pseudo-random hypergraphs}. When appropriately defined, we expect that $\mathcal{G}_{K}^{\textup{pro-nil}}$ is isomorphic to this pro-nilpotent Rado group, at least for $K$ as in Conjecture \ref{cNil}.

In this sense, the final goal of this line of investigation is a substantial strenghtening of the nilpotent case of Shafarevich's celebrated resolution of the inverse Galois problem for solvable groups \cite[Thm. 9.6.1]{Neukirch} by giving an explicit combinatorial description of $\mathcal{G}_K^{\text{pro-nil}}$ in terms of decorated Rado hypergraphs. This would give full justice to the intuition that the pro-nilpotent closure of a number field is a combinatorial object rather than an arithmetical one, containing relatively little information about the number field itself, a point of view that is consistent with Hoshi's counterexamples to the pro-$\ell$ section conjecture \cite{Hoshi}.  

\subsection{Overview of the paper}
The main theory is developed in the even-numbered sections. The odd-numbered sections give the necessary technical modifications to treat the case $\ell = 2$ and can be skipped upon first reading. 

Section \ref{decorated odd graphs} introduces the category of decorated graphs, which is the appropriate analogue of the Rado graph in our situation, and we prove the fundamental properties of this category using the back-and-forth method. Section \ref{Rado group odd primes} abstracts the key group-theoretic properties satisfied by our Galois groups. This culminates in Theorem \ref{weakly Rado groups are isomorphic}, which shows that all (weakly) Rado groups are isomorphic. We apply our abstract framework to Galois groups of quadratic fields in Section \ref{section: Galois groups l odd}. We show how one may naturally attach a decorated graph to such Galois groups. We end Section \ref{section: Galois groups l odd} by proving that the resulting decorated graphs are weakly Rado by an application of the Chebotarev density theorem. Section \ref{sProof} wraps up the paper by combining the results from the previous sections to deduce our main theorems.

\subsection*{Acknowledgments} 
We warmly thank Hendrik Lenstra for numerous insightful conversations on $2$-nilpotent Galois groups and for summarizing the work \cite{Berrevoets} for us. Many thanks go to Fedor Bogomolov for answering our questions about his birational anabelian program and for providing us context and references. We are very grateful to Florian Pop for answering many questions about anabelian geometry and for providing numerous references that considerably helped us putting our results in context. We thank Alexander Ivanov for a stimulating email exchange.

This work was conceived when both authors were guests at the MPIM Bonn to whom goes our deepest gratitude for excellent working conditions and an inspiring atmosphere. We thank Harry Smit for numerous conversations and all the participants of the anabelian geometry seminar at MPIM in Autumn 2020. We thank Hershy Kisilevski for valuable references on the minimal ramification problem.

\subsection{Notation}
Fix an algebraic closure $\overline{\Q}$ throughout this paper. If $K$ is a number field (always taken inside our fixed algebraic closure $\overline{\Q}$), we write
\begin{itemize}
\item $\mathcal{O}_K$ for the ring of integers of $K$;
\item $\mathcal{O}_K^\ast$ for the units of the ring of integers of $K$;
\item $\Omega_K$ for the set of places of $K$;
\item $\mathcal{G}_K$ for the absolute Galois group of $K$;
\item $K^m$ for the compositum of all extensions $L/K$ inside $\overline{\Q}$ such that $L$ is Galois and nilpotent of class at most $m$;
\item $\mathcal{G}_K^m := \Gal(K^m/K)$. We will sometimes write $\mathcal{G}_K^{\text{ab}}$ for $\mathcal{G}_K^1$.
\end{itemize}
Some of the above notations will also be used for local fields. All our cohomology groups are to be interpreted as profinite group cohomology. If $\mathcal{G}$ is a profinite group, then $\mathcal{G}(\ell)$ denotes the maximal pro-$\ell$ quotient of $\mathcal{G}$.

\section{Rado graphs} 
\label{decorated odd graphs}
This section abstracts the key combinatorial tools in the proof of Theorem \ref{Thm: main2} for odd $\ell$. Since $\ell$ is odd, we may fix once and for all an isomorphism of rank $1$ free modules over $\mathbb{Z}_\ell$
$$
\log_\ell: 1+ \ell \cdot \mathbb{Z}_\ell \to \mathbb{Z}_{\ell}.
$$
Recall that this induces an isomorphism between $1 + \ell^{n+1} \cdot \mathbb{Z}_\ell$ and $\ell^n \cdot \mathbb{Z}_\ell$. Hence to compute $\log_\ell(1 + \ell \cdot \alpha)$ modulo $\ell^n$, we only need to know $\alpha$ modulo $\ell^n$. As such, whenever we need to know the value of the function $\log_\ell$ modulo $\ell^n$, we will tolerate that the argument is defined only modulo $\ell^{n+1}$ and is not a priori an element of $\mathbb{Z}_\ell$. We will now define the central objects of this paper.

\begin{definition}
A \emph{decorated graph} is a $4$-tuple $\Phi := (\Omega, f, g, \lambda)$ such that
\begin{itemize}
\item $\Omega$ is a countably infinite set;
\item $f: \Omega \to \mathbb{Z}_{\geq 1} \cup \{\infty\}$ is a function such that the fiber of $\{\infty\}$ has size $2$. We denote these two elements as $v_\ell(1), v_\ell(2)$ and we write $\Omega_{\textup{fin}} := f^{-1}(\mathbb{Z}_{\geq 1})$;
\item the symbol $g$ denotes a choice, for each $v \in \Omega_{\textup{fin}}$, of a generating element $g(v) \in (\Z/\ell^{f(v)}\Z)^\ast$;
\item the symbol $\lambda$ denotes a choice of an element $\lambda(v, w) \in \mathbb{Z}/\ell^{f(v, w)}\mathbb{Z}$ for each pair $(v, w)$ not both in $\{v_\ell(1), v_\ell(2)\}$, where we define $f(v, w) := \min(f(v), f(w))$. Finally, we demand that 
$$ 
\lambda(v_\ell(1), w) \equiv \log_\ell(1 + g(w) \cdot \ell^{f(w)}) \bmod \ell^{f(w)}
$$
for each $w \in \Omega_{\textup{fin}}$ and that
$$
\lambda(w, v_\ell(i)) = 0
$$
for each $w \in \Omega_{\textup{fin}}$ and $i \in \{1, 2\}$. 
\end{itemize}
It will be convenient to extend the notation $\mathbb{Z}/\ell^n\mathbb{Z}$ to the case $n := \infty$, in which case this is to be interpreted as $\Z_\ell$.

An \emph{isomorphism} between two decorated graphs $\Phi := (\Omega, f, g, \lambda), \Phi' := (\Omega', f', g', \lambda')$ consists of a bijection $\phi: \Omega \to \Omega'$ preserving all these data. Namely, for each $v, w \in \Omega$ and each $x \in \Omega_{\textup{fin}}$ one has that
$$
f'(\phi(v)) = f(v), \quad \lambda'(\phi(v), \phi(w)) = \lambda(v, w), \quad g'(\phi(x)) = g(x).
$$
\end{definition}

We will sometimes refer to $\lambda$ as the labels or the labeling of the decorated graph. We now define what it means for a decorated graph to be a \emph{Rado graph}. 

\begin{definition}
\label{def:Rado}
A decorated graph $(\Omega, f, g, \lambda)$ is said to be \emph{Rado} in case for every finite subset $S \subseteq \Omega$ containing $\{v_\ell(1), v_\ell(2)\}$ the following holds. Given the following choices:
\begin{itemize}
\item an element $n \in \mathbb{Z}_{\geq 1}$;
\item a generator $\alpha$ of $\mathbb{Z}/\ell^n\mathbb{Z}$;
\item a pair
$$
((\lambda_s(1))_{s \in S}, (\lambda_s(2))_{s \in S}),
$$
where $\lambda_s(i) \in \mathbb{Z}/\ell^{\min(n, f(s))}\mathbb{Z}$ for each $s \in S$ and each $i \in \{1, 2\}$, where $\lambda_{v_\ell(1)}(2)$ equals the reduction of $\log_\ell(1 + \alpha \cdot \ell^n)$ modulo $\ell^n$ and where $\lambda_{v_\ell(i)}(1) = 0$ for $i \in \{1, 2\}$,
\end{itemize}
there exists $v \in \Omega - S$ with $f(v) = n$, $g(v) = \alpha$ and
$$
\lambda(v, s) = \lambda_s(1), \quad \lambda(s, v) = \lambda_s(2)
$$
for each $s \in S$. 
\end{definition}

In the context of our application, we need two relaxations of the notion of Rado. We present these relaxations in non-decreasing amount of generalization. We start by defining the \emph{scaling} of a graph. 

Given $\gamma_\bullet := (\gamma_v)_{v \in \Omega_{\text{fin}}} \in \prod_{v \in \Omega_{\text{fin}}} (\mathbb{Z}/\ell^{f(v)}\mathbb{Z})^\ast$ and a decorated graph $(\Omega, f, g, \lambda)$, we define $\gamma_\bullet \cdot \lambda$ to be the label given by
$$
(\gamma_\bullet \cdot \lambda)(v, w) := 
\begin{cases}
\gamma_v \cdot \lambda(v, w) & \text{if } v \in \Omega_{\text{fin}} \\
\lambda(v, w) & \text{if } v \not \in \Omega_{\text{fin}}.
\end{cases}
$$
It is readily verified that this defines an action on the space of decorated graphs.

\begin{definition} 
\label{Rado up to scaling}
A decorated graph $(\Omega, f, g, \lambda)$ is said to be \emph{Rado up to scaling} if there exists $\gamma_\bullet \in \prod_{v \in \Omega_{\textup{fin}}} (\mathbb{Z}/\ell^{f(v)}\mathbb{Z})^\ast$ such that $(\Omega, f, g, \gamma_\bullet \cdot \lambda)$ is Rado. 
\end{definition}

We now give what at first glance may look like a further generalization of the previous definition. 

\begin{definition} 
\label{Weakly Rado graphs}
A decorated graph $(\Omega, f, g, \lambda)$ is said to be \emph{weakly Rado} in case for every finite subset $S \subseteq \Omega$ containing $\{v_\ell(1), v_\ell(2)\}$ the following holds. Given the following choices:
\begin{itemize}
\item an element $n \in \mathbb{Z}_{\geq 1}$;
\item a generator $\alpha$ of $\mathbb{Z}/\ell^n\mathbb{Z}$;
\item a pair
$$
((\lambda_s(1))_{s \in S}, (\lambda_s(2))_{s \in S}),
$$
where $\lambda_s(i) \in \mathbb{Z}/\ell^{\min(n, f(s))}\mathbb{Z}$ for each $s \in S$ and each $i \in \{1, 2\}$, where $\lambda_{v_\ell(1)}(2)$ equals the reduction of $\log_\ell(1 + \alpha \cdot \ell^n)$ modulo $\ell^n$ and where $\lambda_{v_\ell(i)}(1) = 0$ for $i \in \{1, 2\}$,
\end{itemize}
then there exists $v \in \Omega - S$ and $\gamma \in (\mathbb{Z}/\ell^n\mathbb{Z})^\ast$ such that $f(v) = n$, $g(v) = \alpha$ and
$$
\lambda(v, s) =\gamma \cdot \lambda_s(1), \quad \lambda(s, v) = \lambda_s(2)
$$
for each $s \in S$.
\end{definition}

We now prove that the last two definitions are equivalent. 

\begin{proposition}
\label{rado up to scaling if and only if weakly rado}
A decorated graph $(\Omega, f, g, \lambda)$ is Rado up to scaling if and only if it is weakly Rado.
\end{proposition}

\begin{proof}
If a decorated graph is Rado up to scaling, then it is also weakly Rado. Therefore it suffices to prove the converse. We will divide the argument in two steps.

\emph{Step 1:} We define $\mathcal{C}$ to be the set of $5$-tuples $(S, n, \alpha, (\lambda_s(1))_{s \in S}, (\lambda_s(2))_{s \in S})$ such that 
\begin{itemize}
\item $S$ is a finite subset of $\Omega$ containing $\{v_\ell(1), v_\ell(2)\}$;
\item $n \in \Z_{\geq 1}$;
\item $\alpha$ belongs to $(\mathbb{Z}/\ell^n\mathbb{Z})^\ast$;
\item $\lambda_s(i) \in \mathbb{Z}/\ell^{\min(n, f(s))}\mathbb{Z}$ for each $s \in S$ and each $i \in \{1, 2\}$, where 
\[
\lambda_{v_\ell(1)}(2) \equiv \log_\ell(1 + \alpha \cdot \ell^n) \bmod \ell^n
\]
and $\lambda_{v_\ell(i)}(1) = 0$ for $i \in \{1, 2\}$.
\end{itemize}
Definition \ref{Weakly Rado graphs} says that a graph is weakly Rado in case there exists a function 
$$
r: \mathcal{C} \to \Omega_{\text{fin}}
$$
such that for each $\mathbf{c} = (S, n, \alpha, (\lambda_s(1))_{s \in S}, (\lambda_s(2))_{s \in S}) \in \mathcal{C}$, there exists an element
$$
\gamma(\mathbf{c}) \in (\mathbb{Z}/\ell^n\mathbb{Z})^\ast
$$
such that $r(\mathbf{c}) \not \in S$ and furthermore
$$
f(r(\mathbf{c})) = n, \quad g(r(\mathbf{c})) = \alpha, \quad \gamma(\mathbf{c}) \cdot \lambda(r(\mathbf{c}), s) = \lambda_s(1), \quad \lambda(s, r(\mathbf{c})) = \lambda_s(2).
$$
We now make the following claim.

\emph{Claim:} For a weakly Rado graph, there exists an \emph{injective} map $r$ as above.

\emph{Proof of claim:} We start by proving that Definition \ref{Weakly Rado graphs} automatically implies the following reinforced version of itself, namely that there exist \emph{infinitely} many vertices $v$ as in the conclusion of Definition \ref{Weakly Rado graphs}. Indeed, if there were only finitely many, let us call $T$ the set of such vertices $v$. Then $S \cup T$ is still a finite set. We now apply Definition \ref{Weakly Rado graphs}, with the same $n$, same $\alpha$ and with any extension of the labels $\lambda$ to $S \cup T$, and reach a contradiction. Now that we know this, we proceed to construct the claimed injective map $r$.

Let $\Omega := \{v_\ell(1), v_\ell(2)\} \cup \{w_1, \ldots, w_h, \ldots\}$ be any enumeration of $\Omega$ (recall that $\Omega$ is countable). We will proceed by induction on $h$, making sure that we have defined $r$ on all elements of $\mathcal{C}$ having as first coordinate a subset of $\{v_\ell(1), v_\ell(2)\} \cup \{w_i : i \leq h\}$. This eventually covers all finite subsets and hence all elements of $\mathcal{C}$.

\emph{Base case:} This is the case where we take $S := \{v_\ell(1), v_\ell(2)\}$. Now observe that once we fix also $n$, then there are only finitely many elements of $\mathcal{C}$ having $S$ as first coordinate and $n$ as second coordinate. On the other hand vertices (i.e. elements of $\Omega$) with different values of $f$ are clearly distinct. Hence the observation that we always have infinitely many vertices $v$ satisfying the conclusion of Definition \ref{Weakly Rado graphs} ensures that we can avoid using the same $v$ twice. Therefore $r$ can be made injective when restricted to elements in $\mathcal{C}$ whose first coordinate is equal to $\{v_\ell(1), v_\ell(2)\}$. 

\emph{Induction step:} Suppose we have constructed $r$ on all the elements of $\mathcal{C}$ having first coordinate a subset of $\{v_1(1), v_\ell(2)\} \cup \{w_i : i \leq h\}$ containing $\{v_\ell(1), v_\ell(2)\}$. We now proceed to extend $r$ to each element of $\mathcal{C}$ whose first coordinate is a subset of $\{v_1(1), v_\ell(2)\} \cup \{w_i : i \leq h + 1\}$ containing $\{v_\ell(1), v_\ell(2)\}$: in other words we need to define it for all the elements of $\mathcal{C}$ whose first coordinate is a subset of $\{v_1(1), v_\ell(2)\} \cup \{w_i : i \leq h + 1\}$ containing $w_{h + 1}$. 

Again we make the key observation that the image of $r$ constructed so far can possibly contain only finitely many vertices with a given value of $f$: there are only finitely many subsets considered so far, and only finitely many elements of $\mathcal{C}$ have some given first two coordinates. And vertices with different values of $f$ are certainly distinct. Hence, once more, the fact that Definition \ref{Weakly Rado graphs} holds also with infinitely many $v$ ensures that we can extend $r$ in an injective fashion, invoking for each element of $\mathcal{C}$ all new vertices not used before.

\emph{Step 2:} Let $r$ be an injective map whose existence is proved in \emph{Step 1} of the argument. We now define an element
$$
\gamma_\bullet \in \left(\prod_{v \in \Omega_{\text{fin}}}\mathbb{Z}/\ell^{f(v)}\mathbb{Z}\right)^\ast
$$
in the following manner. We put
$$
\gamma_v := 
\begin{cases}
1 & \text{if } v \not \in \text{im}(r) \\
\gamma(r^{-1}(v)) & \text{if } v \in \text{im}(r).
\end{cases}
$$
This makes sense, because $r$ is injective. We conclude that the graph $(\Omega, f, g, \gamma_\bullet \cdot \lambda)$ is weakly Rado, since for any point $P \in \mathcal{C}$ we can now choose $r(P)$ to fulfill the requirements of Definition \ref{def:Rado}.  
\end{proof}

We now employ the back-and-forth method from model theory to show that any two Rado graphs are isomorphic. 

\begin{theorem}
\label{Thm:two Rado graphs are isomorphic}
Any two Rado graphs are isomorphic. 
\end{theorem}

\begin{proof}
Let $\{v_\ell(1), v_\ell(2)\}, \{v_\ell(1)', v_\ell(2)'\}$ be the fibers above $\infty$ in respectively $\Omega, \Omega'$. Let us enumerate $\Omega$ and $\Omega'$ as
$$
\Omega := \{v_\ell(1), v_\ell(2)\} \cup \{w_1, \ldots, w_n, \ldots\},
$$
and
$$
\Omega' := \{v_\ell(1)', v_\ell(2)'\} \cup \{w_1', \ldots, w_n', \ldots \}.
$$
We now construct an isomorphism $\phi: \Omega \to \Omega'$ as follows. For each non-negative integer $r$ we construct a triple
$$
(T_r, T_r', \phi_r)
$$
with $T_r \subseteq \Omega, T_r' \subseteq \Omega'$ and with $\phi_r$ a map from $T_r$ to $T_r'$ such that the following properties hold:
\begin{enumerate}
\item[(P1)] $T_i \subseteq T_j, T_i' \subseteq T_j'$ for $i \leq j$.
\item[(P2)] The map $\phi_r: T_r \to T_r'$ is a bijection intertwining $f, g, \lambda$ with $f', g', \lambda'$.
\item[(P3)] The sets $T_0, T_0'$ coincide with respectively $\{v_\ell(1), v_\ell(2)\}, \{v_\ell(1)', v_\ell(2)'\}$, and furthermore $\phi_0(v_\ell(i)) = v_\ell(i)'$ for each $i \in \{1, 2\}$.
\item[(P4)] We have $\phi_j|_{T_i} = \phi_i$ for $i \leq j$.
\item[(P5)] If $r = 2k+1$ is an odd, positive integer, then $w_1, \ldots, w_k$ are all in $T_r$. If $s = 2k$ is an even positive integer, then $w_1', \ldots, w_k'$ are in $T_s'$.
\end{enumerate}

Let us first prove that if such a sequence $((T_r, T_r', \phi_r))_{r \in \mathbb{Z}_{\geq 0}}$ exists, then the corresponding decorated graphs are isomorphic. Indeed, thanks to the first part of property (P5), we have that $\Omega$ is the union of the sets $T_r$ as $r$ runs through the non-negative integers. Furthermore, upon invoking property (P4), we find that the system of maps $\phi_r$ glues to a well-defined map $\phi: \bigcup_{r \geq 0} T_r = \Omega \to \Omega'$. Thanks to the second part of property (P5), the map $\phi$ is surjective. Furthermore, by property (P2), the map $\phi$ is injective and intertwines $f, g, \lambda$ and $f', g', \lambda'$. Therefore $\phi$ is a decorated graph isomorphism. Hence we have now reduced the theorem to proving the following claim.

\emph{Claim:} A sequence $((T_r, T_r', \phi_r))_{r \in \mathbb{Z}_{\geq 0}}$ as above exists.

\emph{Proof of Claim:} We proceed by induction on $r$. 

\emph{Base case:} This is $r := 0$. We are forced to put $\phi_0(v_\ell(i)) = v_\ell(i)'$ for each $i \in \{1, 2\}$. This preserves all the labelings thanks to the definition of a decorated graph.

\emph{Inductive step:} Suppose we have constructed the sequence up until a non-negative integer $r$. We now define $(T_{r + 1}, T_{r + 1}', \phi_{r + 1})$. Let us distinguish two cases.

\emph{Case 1:} Suppose that $r$ is even. Let us write $r + 1=2k+1$. Suppose first that $w_k$ is not in $T_r$ already. Then we define
$$
T_{r + 1} := T_r \cup \{w_k\}.
$$
Next, using that $(\Omega', f', g', \lambda')$ is Rado, we can find $v \in \Omega'$ outside of $T_r'$ such that $f(w_k) = f'(v), g(w_k) = g'(v)$ and furthermore
$$
\lambda(w_k, t) = \lambda'(v, \phi_r(t)), \quad \lambda(t, w_k) = \lambda'(\phi_r(t), v)
$$
for each $t \in T_r$. We now define $T_{r + 1}':=T_r' \cup \{v\}$ and extend $\phi_r$ to $\phi_{r + 1}$ by declaring $\phi_{r + 1}(w_k) = v$. 

Suppose next that $w_k$ is already in $T_r$. Then we declare $T_{r + 1} := T_r$, $T_{r + 1}' := T_r'$ and $\phi_{r + 1} := \phi_r$. In either case, $(T_{r + 1}, T_{r + 1}', \phi_{r + 1})$ respects all the requirements.

\emph{Case 2:} Suppose now that $r$ is odd. Let us write $r + 1 = 2k$. We proceed as in Case 1, but with the roles of $(\Omega, f, g, \lambda)$ and $(\Omega', f', g', \lambda')$ interchanged.

This completes the inductive step, thus proving the claim and hence the theorem.
\end{proof}

\begin{remark} 
\label{rmk: for Rado graphs the cyclotomic places is reconstructed}
Observe that isomorphisms of decorated graphs preserve the set $\{v_\ell(1), v_\ell(2)\}$, but a priori there is no distinction between the two elements. However, when the two graphs are \emph{Rado}, it is not difficult to check that an isomorphism has to preserve the indexing.   
\end{remark}

Rado graphs are in some sense abundant. As the next theorem formalizes, almost any graph is Rado. Fixing $\Omega$ and $f: \Omega \to \mathbb{Z}_{\geq 1} \cup \{\infty\}$, we see that the set of graphs (having as first two coordinates $(\Omega, f)$) is now parametrized by the choices of $(g, \lambda)$, which is a vector in
$$
\left(\prod_{v \in \Omega_{\text{fin}}} (\mathbb{Z}/\ell^{f(v)}\mathbb{Z})^\ast \right) \times \left(\prod_{(v, w): v \neq w} \mathbb{Z}/\ell^{f(v, w)}\mathbb{Z}\right).
$$
Here the pairs $(v,w)$ are running only among those where $v \neq v_{\ell}(1)$ and $w \not \in \{v_{\ell}(1),v_{\ell}(2)\}$, as in all other cases the labels are already determined by the rest of the data. 

This is a compact topological group, where we view each factor with the discrete topology and the total group with the product topology. Hence it comes with a Haar measure. 

\begin{theorem} 
\label{prob=1}
Suppose that $f^{-1}(n)$ is infinite for all $n \in \mathbb{Z}_{\geq 1}$. Then 
$$
\mathbb{P}((\Omega, f, g, \lambda) \textup{ is Rado}) = 1.
$$
In particular, two randomly selected decorated graphs are isomorphic with probability $1$. 
\end{theorem}

\begin{proof}
Observe that the set $\mathcal{C}$ introduced in the proof of Proposition \ref{rado up to scaling if and only if weakly rado} is countable. We also keep the notation for $r$ as used in that proof. Now the property of being not Rado is the countable union, parametrized by $\mathcal{C}$, of the events that one can not define $r(P)$ for $P$ in $\mathcal{C}$. Since this union is countable, it is therefore sufficient to prove that each of these events occur with probability $0$, which is readily verified.
\end{proof}

\section{\texorpdfstring{Rado graphs for $\ell = 2$}{Rado graphs for 2}} 
\label{decorated graphs at 2}
This section adapts the combinatorial tools of Section \ref{decorated odd graphs} to the case $\ell := 2$. Let us start by fixing once and for all an isomorphism of rank $1$ free modules over $\Z_2$
$$
\log_2: 1 + 4 \cdot \Z_2 \to \Z_2.
$$
This induces an isomorphism between $1 + 2^{n + 2} \cdot \mathbb{Z}_2$ and $2^n \cdot \Z_2$. Hence whenever we need to find the value of $\log_2(1 + 4 \cdot \alpha)$ modulo $2^n$, we only need to know the value of $\alpha$ modulo $2^n$.

Our next goal is to define decorated graphs at $2$: these are vastly similar to the graphs appearing for $\ell$ odd with only a slightly more intricate relationship between $f$, $g$ and the labels. These technical differences in our application are motivated by the fact that $\mathbb{Z}_{\ell}^{\ast}$ is pro-cyclic for $\ell$ odd and is not pro-cyclic for $\ell=2$, a difference that also showed up in the slightly modified definition of $\log_2$ above.

\begin{definition} 
\label{def: decorated graphs at 2}
We define the category of \emph{decorated graphs} in the same way as Definition \ref{def: decorated graphs at 2}, except that we make the following changes to the assumptions on the labels $\lambda$.

For $\ell = 2$, the symbol $\lambda$ denotes a choice of an element $\lambda(v, w) \in \mathbb{Z}/2^{f(v, w)}\mathbb{Z}$ for each pair $(v, w)$ not both in $\{v_2(1), v_2(2)\}$. We demand that $\lambda$ satisfies the following axioms:
\begin{itemize}
\item $\lambda(v_2(2), w) \equiv 1 \bmod 2 \Longleftrightarrow f(w) = 1$ for all $w \in \Omega_{\textup{fin}}$;
\item $\lambda(v_2(1), w) \equiv \log_2(1 + g(w) \cdot 2^{f(w)}) \bmod 2^{f(w)}$ for all $w \in \Omega_{\textup{fin}}$ with $f(w) \geq 2$;
\item $\lambda(w, v_2(i)) = 0$ for each $w \in \Omega_{\textup{fin}}$ and $i \in \{1, 2\}$. 
\end{itemize}
\end{definition}

Decorated graphs at $2$ are a strictly bigger class of graphs with respect to the one appearing in our application, which is a class of graphs manifesting additional constraints essentially due to quadratic reciprocity. The bigger class given in Definition \ref{def: decorated graphs at 2} will correspond to a broader class of groups and it is for this class that it is most natural to attach a graph. It is for this reason that we consider this slightly more general class of objects. We now define what it means for a decorated graph to be a \emph{decorated reciprocity graph}, these are precisely the graphs that will show up in our application.  

\begin{definition}
Let $\Phi := (\Omega, f, g, \lambda)$ be a decorated graph. We say that $\Phi$ is a \textup{decorated reciprocity graph} if we have that
$$
\lambda(v, w) \equiv \lambda(w, v) + \lambda(v_2(2), v) \cdot \lambda(v_2(2), w) \bmod 2
$$
for all $v, w \in \Omega_{\textup{fin}}$.
\end{definition}

We can now define what it means for a decorated reciprocity graph to be a \emph{Rado graph}.

\begin{definition}
\label{def:Rado at 2}
A decorated reciprocity graph $(\Omega, f, g, \lambda)$ is said to be \emph{Rado} if for every finite subset $S \subseteq \Omega$ containing $\{v_2(1), v_2(2)\}$ the following holds. Given the following choices:
\begin{itemize}
\item an element $n \in \mathbb{Z}_{\geq 1}$;
\item a generator $\alpha$ of $\mathbb{Z}/2^n\mathbb{Z}$;
\item a pair 
$$
((\lambda_s(1))_{s \in S}, (\lambda_s(2))_{s \in S}),
$$
where $\lambda_s(i) \in \mathbb{Z}/2^{\min(n, f(s))}\mathbb{Z}$ for each $s \in S$ and each $i \in \{1, 2\}$, satisfying the following requirements:
\begin{itemize}
\item we have $n = 1$ if and only if $\lambda_{v_2(2)}(2) \equiv 1 \bmod 2$;
\item if $n \geq 2$, we have that $\lambda_{v_2(1)}(2)$ equals the reduction of $\log_2(1 + \alpha \cdot 2^n)$ modulo $2^n$, while $\lambda_{v_2(i)}(1) = 0$ for $i \in \{1, 2\}$;
\item we have $\lambda_s(1) \equiv \lambda_s(2) + \lambda(v_2(2),s) \cdot \lambda(v_2(2), v) \bmod 2$ for each $s \in S - \{v_2(1), v_2(2)\}$,
\end{itemize}
\end{itemize}
then there exists $v \in \Omega - S$ with $f(v) = n$, $g(v) = \alpha$ and
$$
\lambda(v, s) = \lambda_s(1), \quad \lambda(s, v) = \lambda_s(2)
$$
for each $s \in S$. 
\end{definition}

Scaling of decorated graphs are defined precisely in the same way as for $\ell$ odd. We can make now verbatim the same definition of graphs that are \emph{Rado up to scaling}. 

\begin{definition} 
\label{def: Rado up to scaling at 2}
A decorated reciprocity graph $(\Omega, f, g, \lambda)$ is said to be \emph{Rado up to scaling} if there exists $\gamma_\bullet \in \prod_{v \in \Omega_{\textup{fin}}} (\mathbb{Z}/2^{f(v)}\mathbb{Z})^\ast$ such that $(\Omega, f, g, \gamma_\bullet \cdot \lambda)$ is Rado.
\end{definition}

Likewise, just throwing the reciprocity constraints into the mix, we readily have the analogous definition of \emph{weakly Rado graphs}. 

\begin{definition} 
\label{def: weakly Rado at 2}
A decorated reciprocity graph $(\Omega, f, g, \lambda)$ is said to be \emph{weakly Rado} if for every finite subset $S \subseteq \Omega$ containing $\{v_2(1), v_2(2)\}$ the following holds. Given the choices as in Definition \ref{def:Rado at 2}, there exists $v \in \Omega - S$ and $\gamma \in (\mathbb{Z}/2^n\mathbb{Z})^\ast$ such that $f(v) = n$, $g(v) = \alpha$ and
$$
\lambda(v, s) =\gamma \cdot \lambda_s(1), \quad \lambda(s, v) = \lambda_s(2)
$$
for each $s \in S$.
\end{definition}

\begin{remark}
Observe that Rado, weakly Rado and Rado up to scaling graphs are decorated \emph{reciprocity} graphs by definition. The reason for this particular setup for $\ell = 2$ is that decorated graphs are only accessory to defining the right category of graphs but are not appearing in the actual application to Galois groups. 
\end{remark}

Making straightforward changes to the proof of Proposition \ref{rado up to scaling if and only if weakly rado} we have the following fact. 

\begin{proposition} 
\label{rado up to scaling if and only if weakly rado at 2}
A decorated reciprocity graph $(\Omega, f, g, \lambda)$ is Rado up to scaling if and only if it is weakly Rado.  
\end{proposition}

Next, with straightforward modifications to the back-and-forth argument of Theorem \ref{Thm:two Rado graphs are isomorphic}, we obtain the main theorem of this section.

\begin{theorem} 
\label{Thm:two Rado graphs are isomorphic for 2}
Any two Rado graphs are isomorphic.     
\end{theorem}

Finally, we verify that Rado graphs are abundant among decorated reciprocity graphs also for $\ell = 2$. We spell out the details mostly to go through the minor technical difference that the space of decorated reciprocity graphs is naturally a \emph{coset} of a group rather than a group itself. Namely, fixing $f: \Omega \to \Z_{\geq 1} \cup \{\infty\}$, we see that the set of decorated graphs (having as first two coordinates $(\Omega, f)$) is now parametrized by the choices of $(g, \lambda)$, which is a vector in 
\begin{align}
\label{eDecSpace}
\left(\prod_{v \in \Omega_{\text{fin}}} (\mathbb{Z}/2^{f(v)}\mathbb{Z})^\ast \right) \times \left(\prod_{(v, w): v \neq w} \mathbb{Z}/2^{f(v, w)}\mathbb{Z}\right).
\end{align}
Here the pairs $(v,w)$ are running among those where $w \not \in \{v_2(1), v_2(2)\}$ and furthermore $v \neq v_2(1)$ if $f(w) = 1$. Indeed, the rest of the labels are determined by the existing data. Observe that not all pairs $(g, \lambda)$ correspond to decorated graphs, as we have not yet encoded the condition that $f(w) = 1$ if and only if $\lambda(v_2(2), w)$ is odd: we will encode this condition directly with the reciprocity conditions (in fact this condition in later sections does ultimately reflect another instance of quadratic reciprocity). 

We endow each factor in equation (\ref{eDecSpace}) with the discrete topology and the total group with the product topology. The set of decorated reciprocity graphs is a closed subspace of this product space. More precisely, consider the group homomorphism
$$
\left(\prod_{v \in \Omega_{\text{fin}}} (\mathbb{Z}/2^{f(v)}\mathbb{Z})^\ast \right) \times \left(\prod_{(v, w): v \neq w} \mathbb{Z}/2^{f(v, w)}\mathbb{Z}\right) \to \prod_{(v,w): v \neq w} \mathbb{F}_2,
$$
given by sending a vector $((\gamma_v), (\lambda(v, w)))$ to $(\lambda(v, w) + \lambda(w, v) \bmod 2)$, when both $v, w \in \Omega_{\text{fin}}$, and by sending it to $\lambda(v_2(2), w) \bmod 2$, when $w \in \Omega_{\text{fin}}$. Recall that all the other pairs involving the fiber above $\infty$ are not in the set of indices of the direct product.  

The set of decorated reciprocity graphs equals the counterimage of the indicator function that $f(v) = f(w) = 1$ for $v, w \in \Omega_{\text{fin}}$ and the indicator function that $f(w) = 1$ for the remaining pairs. As such the set of decorated reciprocity graphs is not only closed, but it is the coset of a closed (hence compact) subgroup $K$ (the kernel of this continuous homomorphism). We push forward the Haar measure along any isomorphism given by translation with a given point $g_0$ in the coset. The choice of $g_0$ in the coset does not change the measure, as it comes from shifting $g_0$ with an element of $K$ and the Haar measure is invariant under this operation. With straightforward modifications to the proof, we now have the analogue of Theorem \ref{prob=1}.

\begin{theorem} 
\label{prob=1 for 2}
Suppose that $f^{-1}(n)$ is infinite for all $n \in \mathbb{Z}_{\geq 1}$. Then 
$$
\mathbb{P}((\Omega, f, g, \lambda) \textup{ is Rado}) = 1.
$$
In particular, two randomly selected decorated reciprocity graphs are isomorphic with probability $1$.   
\end{theorem}

\section{Rado groups} 
\label{Rado group odd primes}
This section abstracts the key group-theoretic tools we employ in the part of the proof of Theorem \ref{Thm: main2} devoted to odd primes. 

Let $\ell$ be an odd prime number. Let $(\Omega, f, g)$ be a triple as in Section \ref{decorated odd graphs} satisfying the axioms of a decorated graph involving only these three variables. 

\begin{definition} 
\label{def: Rado groups} 
A $2$-nilpotent pro-$\ell$ group on $(\Omega, f, g)$ is a pair $\mathcal{S} := (\mathcal{G}, \pi_\bullet)$, where
\begin{enumerate}
\item[(C1)] $\mathcal{G}$ is a $2$-nilpotent pro-$\ell$ group;
\item[(C2)] the map 
$$
\pi_\bullet := (\pi_v)_{v \in \Omega}: \mathcal{G} \to \prod_{v \in \Omega} \mathbb{Z}/\ell^{f(v)}\mathbb{Z}
$$
is a continuous group epimorphism such that $\ker(\pi_\bullet) = [\mathcal{G}, \mathcal{G}]$;
\item[(C3)] for each pair of distinct $v, w \in \Omega$, the natural inflation map induced by $(\pi_v, \pi_w)$
$$
H^2(\mathbb{Z}/\ell^{f(v)}\mathbb{Z} \times \mathbb{Z}/\ell^{f(w)}\mathbb{Z}, \mathbb{Z}/\ell^{f(v, w)}\mathbb{Z}) \xrightarrow{\textup{inf}} H^2(\mathcal{G}, \mathbb{Z}/\ell^{f(v, w)}\mathbb{Z})
$$
has image that coincides with the image of the inflation map
$$
\textup{Ext}^1_{\mathbb{Z}_\ell}(\mathbb{Z}/\ell^{f(v)}\mathbb{Z} \times \mathbb{Z}/\ell^{f(w)}\mathbb{Z}, \mathbb{Z}/\ell^{f(v, w)}\mathbb{Z}) \xrightarrow{\textup{inf}} H^2(\mathcal{G}, \mathbb{Z}/\ell^{f(v, w)}\mathbb{Z});
$$
\item[(C4)] we attach to each $v \in \Omega_{\textup{fin}}$ the unique extension class, which we denote by $\theta_v$, in $\textup{Ext}^1_{\mathbb{Z}_\ell}(\mathbb{Z}/\ell^{f(v)}\mathbb{Z}, \mathbb{Z}/\ell^{f(v)}\mathbb{Z})$ corresponding to the exact sequence where lifting $1$ and multiplying it by $\ell^{f(v)} \cdot g(v)$ yields $1$. Let us call $\tilde{\theta}_v$ the inflation of this class to $H^2(\mathcal{G}, \mathbb{Z}/\ell^{f(v)}\mathbb{Z})$ through $\pi_v$. Then we have that 
$$
\pi_{v_\ell(1)} \cup \pi_v + \log_\ell(1+ \ell^{f(v)} \cdot g(v)) \cdot \tilde{\theta}_v
$$
equals $0$ in $H^2(\mathcal{G}, \mathbb{Z}/\ell^{f(v)}\mathbb{Z})$. 
\end{enumerate}
\end{definition}

\begin{remark}
If $\{v, w\} = \{v_\ell(1), v_\ell(2)\}$, then the two cohomology groups 
$$
H^2(\mathbb{Z}/\ell^{f(v)}\mathbb{Z} \times \mathbb{Z}/\ell^{f(w)}\mathbb{Z}, \mathbb{Z}/\ell^{f(v, w)}\mathbb{Z}) = H^2(\mathbb{Z}_\ell \times \mathbb{Z}_\ell, \mathbb{Z}_\ell) \textup{ and } H^2(\mathcal{G}, \mathbb{Z}/\ell^{f(v, w)}\mathbb{Z}) = H^2(\mathcal{G}, \mathbb{Z}_\ell)
$$ 
are to be interpreted with the profinite topology also on the $\mathbb{Z}_\ell$ appearing in the second argument of the $H^2$, unlike the usual discrete topology that one finds in profinite group cohomology.
\end{remark}

To simplify the exposition, we will adopt the convention that $\tilde{\theta}_v = 0$ whenever $f(v) = \infty$.

Let $\mathcal{S}, \mathcal{S}'$ be a pair of $2$-nilpotent pro-$\ell$ groups on respectively $(\Omega, f, g)$ and $(\Omega', f', g')$. We call an isomorphism between them a pair $(\phi, \gamma)$, where $\gamma: \Omega \to \Omega'$ is a bijection intertwining $(f, g)$ with $(f', g')$ and $\phi: \mathcal{G} \to \mathcal{G'}$ is a group isomorphism such that
$$
\pi_v = \pi_{\gamma(v)} \circ \phi.
$$
Our next goal is to attach to each $\mathcal{S}$ a certain labeling $\lambda_\mathcal{S}$, which will turn 
$$
\Phi_\mathcal{S} := (\Omega, f, g, \lambda_\mathcal{S})
$$
into a graph encoding the isomorphism class of $\mathcal{S}$. But we will first take a detour on profinite alternating squares.

\subsection{\texorpdfstring{Intermezzo: the profinite $\wedge^2$}{Intermezzo on the profinite, second exterior algebra}}
\label{intermezzo}
In this subsection we drop the assumption that $\ell$ is odd and we allow also $\ell=2$. Let us recall the definition of the functor $\wedge^2: \text{Ab-Gr} \to \text{Ab-Gr}$, given by sending an abelian group $A$ to
$$
\wedge^2 A := (A \otimes_{\mathbb{Z}} A)/\langle \{a \otimes a : a \in A\} \rangle
$$
and a group homomorphism $f:A_1 \to A_2$ to
$$
\wedge^2(f): \wedge^2 A_1 \to \wedge^2 A_2,
$$
where the map $\wedge^2(f)$ sends an elementary tensor $a \otimes b$ to $f(a) \otimes f(b)$. 

Recall that if $A$, $B$ are two abelian groups, then $b: A \times A \to B$ is said to be \emph{alternating} in case $b(a, a) = 0$ for each $a \in A$. We have that the alternating map
$$
b_{\text{univ}}(A): A \times A \to \wedge^2 A, 
$$
sending $(a, a')$ to $a \wedge a'$ (which is by definition the image of $a \otimes a'$), is \emph{universal} in the sense that given an alternating map 
$$
b: A \times A \to B,
$$
there is a \emph{unique} homomorphism $\phi: \wedge^2 A \to B$ such that $b = \phi \circ b_{\text{univ}}(A)$. We will refer to this as the \emph{universal property} of $\wedge^2$. 

Recall that for any two abelian groups $A$, $B$ with $A$ acting trivially on $B$, one has a split exact sequence
\begin{align}
\label{eExtWedge}
0 \to \text{Ext}_{\mathbb{Z}}^1(A, B) \to H^2(A, B) \to \text{Hom}_{\text{ab.gr.}}(\wedge^2 A, B) \to 0
\end{align}
coming from the universal coefficient theorem. A concrete interpretation of the second map, in terms of equivalence classes of central extensions, comes from lifting and taking commutators: the result does not depend on the choice of the lift thanks to centrality of the extension. This brings us to the next point.

Recall that a group is said to be nilpotent of class at most $2$ if $[G, [G, G]] = \{\text{id}\}$. We have now two functors
$$
\wedge^2 \circ \mathcal{F}^{\text{ab}}, \mathcal{F}_{[,]}: \{\text{nilpotent groups of class at most 2}\} \to \text{Ab-Gr},
$$
where $\mathcal{F}^{\text{ab}}$ is the abelianization functor and where $\mathcal{F}_{[,]}$ sends $G$ to $[G, G]$. We have a \emph{natural transformation} from the first to the second functor, given by the system of maps
$$
\nu_G: \wedge^2 G^{\text{ab}} \to [G, G]
$$
coming from the alternating map described above: lift and take commutators. That this is indeed a natural transformation comes down to the identity $[\phi(g_1), \phi(g_2)] = \phi([g_1, g_2])$ valid for any group homomorphism $\phi$.

When we specialize the exact sequence (\ref{eExtWedge}) to $\Q/\Z =: B$, we obtain that the first term vanishes as $\Q/\Z$ is injective. Therefore we get an identification
\begin{align}
\label{eH2Wedge}
H^2(A, \Q/\Z) = (\wedge^2 A)^\vee,
\end{align}
where $A^\vee := \text{Hom}_{\text{ab.gr.}}(A, \Q/\Z)$. 

Let us now turn to the \emph{profinite} case. Let $\mathcal{A}$ be a profinite abelian group. Let us define
$$
\wedge^2_{\text{prof}} \mathcal{A} := \varprojlim_{U} \wedge^2 \mathcal{A}_U,
$$
where $U$ runs through the open subgroups of $\mathcal{A}$ and where $\mathcal{A}_U := \frac{\mathcal{A}}{U}$. Observe that this comes with a map
$$
b_{\text{univ}}(\mathcal{A})(a_1, a_2) = a_1 \wedge a_2,
$$
from $\mathcal{A} \times \mathcal{A}$ to $\wedge^2_{\text{prof}} \mathcal{A}$, packaging all the $b_{\text{univ}}(\mathcal{A}_U)$ as $U$ runs through the open subgroups of $\mathcal{A}$. Let us prove that this map is universal among profinite abelian groups. 

\begin{proposition} 
\label{Universal property}
Let $\mathcal{A}$, $\mathcal{B}$ be profinite abelian groups and let
$$
b: \mathcal{A} \times \mathcal{A} \to \mathcal{B}
$$
be a continuous alternating map. Then there exists a unique continuous homomorphism $\phi: \wedge^2_{\textup{prof}} \mathcal{A} \to \mathcal{B}$ such that
$$
b = \phi \circ b_{\textup{univ}}.
$$
\end{proposition}

\begin{proof}
We will first treat the special case where $\mathcal{B} = B$ is discrete and finite. Then observe that the counterimage of $0_B$ is an open subset containing $(0, 0)$. Hence, by construction of the product topology, this counterimage contains a product of two open sets $V_1 \times V_2$, both containing $0$. Therefore there exists an open subgroup $V$ such that $b(V \times V) = 0$. It follows that the triple intersection between $V$, the right kernel and the left kernel equals the intersection of finitely many open subgroups (no more than the index of $V$ in $\mathcal{A}$). We have therefore proved that the bilinear maps factors through the projection $\pi_W:\mathcal{A} \to \mathcal{A}_W$ for some open subgroup $W$. 

Hence in the special case where $B$ is a finite, discrete abelian group, we have shown that there exists a map $\phi: \wedge^{2}_{\text{prof}} \mathcal{A} \to B$ such that $b=\phi \circ b_{\text{univ}}$. To see that $\phi$ is unique, observe that the subset
$$
\{a_1 \wedge a_2 : a_1,a_2 \in \mathcal{A}\} \subseteq \wedge^2_{\text{prof}} \mathcal{A}
$$
is a set of topological generators. Indeed, for each open subgroup $U$ of $\mathcal{A}$, it gives a set of generators in $\wedge^2_{\text{prof}} \mathcal{A}_U$. This last property holding for each open subgroup $U$ is equivalent to being a set of topological generators, for any profinite group. Hence, since $\phi(a_1 \wedge a_2) = b(a_1, a_2)$ for all $a_1, a_2 \in \mathcal{A}$, we must have that $\phi$ is unique. Thus we have proved the desired universality property in the special case that $\mathcal{B}$ is finite and discrete.

Let now $\mathcal{B}$ be an arbitrary profinite abelian group. For each open subgroup $U$ of $\mathcal{B}$, we may consider the continuous bilinear alternating map
$$
\pi_U \circ b: \mathcal{A} \times \mathcal{A} \to \mathcal{B}_U.
$$
From the special case already treated, we know there exists a unique continuous homomorphism $\phi_U:\wedge^2_{\text{prof}} \mathcal{A} \to \mathcal{B}_U$ such that
$$
\phi_U \circ b_{\text{univ}}(\mathcal{A}) = \pi_U \circ b.
$$
By the uniqueness of the $\phi_U$ we deduce that they form a compatible system. Hence we obtain a continuous homomorphism $\phi: \mathcal{A} \to \mathcal{B}$ satisfying $\phi_U = \pi_U \circ \phi$ and
$$ 
\pi_U \circ \phi \circ b_{\text{univ}}(\mathcal{A}) = \pi_U \circ b
$$
for each open subgroup $U$ of $\mathcal{A}$. The last equality holding for each open $U$ is equivalent to the equality
$$
\phi \circ b_{\text{univ}}(\mathcal{A}) = b,
$$
as desired. This proves that $\phi$ exists. To see that $\phi$ is unique, we argue precisely as we did in the case that $\mathcal{B}$ was finite, by using that the set $\{a_1 \wedge a_2 : a_1, a_2 \in \mathcal{A}\} \subseteq \wedge^2_{\text{prof}} \mathcal{A}$ is a set of topological generators in $\wedge^2_{\text{prof}} \mathcal{A}$ and $\phi$ is already determined on each of them by $\phi(a_1 \wedge a_2) = b(a_1, a_2)$. Since $\phi$ is continuous, this shows that there is at most one such $\phi$ as desired.
\end{proof}

The universal property established in Proposition \ref{Universal property} also proves that $\wedge^2_{\text{prof}}$ extends to a functor from profinite abelian groups to profinite abelian groups. Given a homomorphism $f:\mathcal{A}_1 \to \mathcal{A}_2$, the alternating continuous bilinear map
$$
b_{\text{univ}}(\mathcal{A}_2) \circ (f, f): \mathcal{A}_1 \times \mathcal{A}_1 \to \wedge^2_{\text{prof}} \mathcal{A}_2
$$
induces by Proposition \ref{Universal property} a unique homomorphism 
$$
\wedge^2_{\text{prof}}(f): \wedge^2_{\text{prof}} \mathcal{A}_1 \to \wedge^2_{\text{prof}} \mathcal{A}_2
$$
such that
$$
\wedge^2_{\text{prof}}(f) \circ b_{\text{univ}}(\mathcal{A}_1)=b_{\text{univ}}(\mathcal{A}_2) \circ (f, f).
$$
This observation allows one to readily verify that $\wedge^2_{\text{prof}}$ is indeed a functor.

Thanks to equation (\ref{eH2Wedge}), we have a natural isomorphism
$$
e_{\mathcal{A}}: \wedge^2_{\text{prof}} \mathcal{A} \to H^2(\mathcal{A}, \Q/\Z)^\vee,
$$
in other words the profinite group $\wedge^2_{\text{prof}} \mathcal{A}$ and the discrete group $H^2(\mathcal{A}, \Q/\Z)$ are Pontryagin duals. Consider the functors
$$
\wedge^2 \circ \mathcal{F}^{\text{ab}}, \mathcal{F}_{[,]}: \{\text{profinite nilpotent groups of class at most 2}\} \to \{\text{profinite abelian groups}\},
$$
where $\mathcal{F}^{\text{ab}}$ is the abelianization functor and where $\mathcal{F}_{[,]}$ sends $G$ to $[G, G]$. Here $[G, G]$ denotes the commutator subgroup in the usual profinite sense, which is by definition the closure of the subgroup generated by elements of the form $ghg^{-1}h^{-1}$. We can also extend the \emph{natural transformation} from the first to the second, given by the system of maps
$$
\nu_{\mathcal{G}}: \wedge^2_{\text{prof}} \mathcal{G}^{\text{ab}} \to [\mathcal{G}, \mathcal{G}],
$$
coming from lifting and taking commutators: this is a continuous alternating map so we get the map $\nu_{\mathcal{G}}$ through Proposition \ref{Universal property}. Observe that $\nu_{\mathcal{G}}$ is surjective since the image of the pure wedges are all the commutators, which are a set of topological generators for the commutator subgroup. 

An alternative proof of this result is to employ the inflation-restriction exact sequence. Since the restriction map from $H^1(\mathcal{G}, \Q/\Z)$ to $H^1([\mathcal{G}, \mathcal{G}], \Q/\Z)$ is trivial, we get an injection
$$
{[\mathcal{G}, \mathcal{G}]}^\vee \to H^2(\mathcal{G}^{\text{ab}}, \Q/\Z),
$$
which composed with $e_{\mathcal{G}^{\text{ab}}}$ gives an injection $$[\mathcal{G}, \mathcal{G}]^\vee \to (\wedge^2_{\text{prof}}\mathcal{G}^{\text{ab}})^\vee.
$$
Direct inspection of the inflation-restriction sequence gives that this is none other than the alternating map
$$
(g_1, g_2) \mapsto \phi([g_1', g_2']),
$$
where $g_1'$ and $g_2'$ are lifts of respectively $g_1$ and $g_2$ to $\mathcal{G}$. Hence the dual of this map is nothing else than the surjection $\nu_{\mathcal{G}}$ explained above. The following propositions will be helpful. 

\begin{proposition} 
\label{inflation of h^2=0}
Let $\mathcal{G}$ be a profinite nilpotent group of class at most $2$. Suppose that the inflation map $H^2(\mathcal{G}^{\textup{ab}}, \Q/\Z) \to H^2(\mathcal{G}, \Q/\Z)$ is the zero map. Then $\nu_{\mathcal{G}}$ is an isomorphism.
\end{proposition} 

\begin{proof}
As explained above, the restriction map from $H^1(\mathcal{G}, \Q/\Z)$ to $H^1([\mathcal{G}, \mathcal{G}], \Q/\Z)$ is the zero map. The assumption on the inflation tells us that the connecting homomorphism in the inflation-restriction sequence is an isomorphism
$$
[\mathcal{G}, \mathcal{G}]^\vee \to H^2(\mathcal{G}^{\text{ab}}, \Q/\Z),
$$
which composed with $e_{\mathcal{G}^{\text{ab}}}^\vee$ therefore yields an isomorphism
$$
[\mathcal{G}, \mathcal{G}]^\vee \to (\wedge^2_{\text{prof}} \mathcal{G}^{\text{ab}})^\vee.
$$
The dual of this map, as explained right before this proof, is precisely $\nu_{\mathcal{G}}$. Hence, since $\nu_{\mathcal{G}}^\vee$ is an isomorphism, so is $\nu_{\mathcal{G}}$, which is the desired conclusion. 
\end{proof}

We call a profinite nilpotent group \emph{free on their abelianization} in case the inflation map $H^2(\mathcal{G}^{\textup{ab}}, \Q/\Z) \to H^2(\mathcal{G}, \Q/\Z)$ is the zero map.

\begin{proposition} 
\label{Enough surjective on abelianization}
Let $\mathcal{G}_1$, $\mathcal{G}_2$ be two profinite nilpotent groups of class at most $2$, both free on their abelianization. Suppose that we have a continuous homomorphism
$$
\phi: \mathcal{G}_1 \to \mathcal{G}_2
$$
that induces an isomorphism on the abelianizations. Then $\phi$ is an isomorphism. 
\end{proposition}

\begin{proof}
We claim that the restriction of $\phi: [\mathcal{G}_1, \mathcal{G}_1] \to [\mathcal{G}_2, \mathcal{G}_2]$ to the commutator subgroups is an isomorphism. Indeed, since $\nu$ is a natural transformation, we have that
$$
\mathcal{F}_{[,]}(\phi) \circ \nu_{\mathcal{G}_1} = \nu_{\mathcal{G}_2} \circ \wedge^2_{\text{prof}}(\mathcal{F}^{\text{ab}}(\phi)).
$$
By assumption $\mathcal{F}^{\text{ab}}(\phi)$ is an isomorphism. Hence $\wedge^2_{\text{prof}}(\mathcal{F}^{\text{ab}}(\phi))$ is also an isomorphism by functoriality of $\wedge_2^{\text{prof}}$. Thanks to Proposition \ref{inflation of h^2=0}, we know that $\nu_{\mathcal{G}_1}, \nu_{\mathcal{G}_2}$ are isomorphisms. We deduce that $\mathcal{F}_{[,]}(\phi)$ is an isomorphism as claimed. 

So far we have shown that $\mathcal{F}_{[,]}(\phi), \mathcal{F}^{\text{ab}}(\phi)$ are both isomorphisms. Hence $\phi$ is an isomorphism, as we now explain (this is a standard fact, but we replicate the proof for completeness). Let us first show that $\phi$ is injective. If $\phi(g) = \text{id}$, then $g$ has to be in the commutator subgroup, because $\mathcal{F}^{\text{ab}}(\phi)$ is injective. Therefore, since $\mathcal{F}_{[,]}(\phi)$ is also injective, we deduce that $g = \text{id}$ as desired. 

It remains to show that $\phi$ is surjective. To this end, take $\gamma \in \mathcal{G}_2$. Because $\mathcal{F}^{\text{ab}}(\phi)$ is surjective, there exists $g \in \mathcal{G}_1$ such that $\phi(g)\gamma^{-1}$ is a commutator. Next, we can always find an element $h$ of the commutator subgroup of $\mathcal{G}_1$ such that $\phi(h) = \phi(g)\gamma^{-1}$, since $\mathcal{F}_{[,]}(\phi)$ is also surjective. Therefore we have $\phi(h^{-1}g) = \gamma$, showing that $\phi$ is surjective. 
\end{proof}

We end this subsection with the following handy fact. We keep the convention used so far that if we have a function
$$
f: \Omega \to \mathbb{Z}_{\geq 1} \cup \{\infty\},
$$
then $f(v, w) := \min(f(v), f(w))$ and the convention that $\mathbb{Z}/\ell^n\mathbb{Z} = \mathbb{Z}_\ell$ when $n:=\infty$. 

\begin{proposition} 
\label{computation of profinite wedge 2}
Let $(\Omega, \leq)$ be a totally ordered set and let $f: \Omega \to \mathbb{Z}_{\geq 1} \cup \{\infty\}$ be a function. Denote by
$$
\mathcal{A} := \prod_{v \in \Omega} \mathbb{Z}/\ell^{f(v)}\mathbb{Z}.
$$
For $v < w$ with $v, w \in \Omega$, denote by
$$
\psi(v, w): \wedge^2_{\textup{prof}} \mathcal{A} \to \mathbb{Z}/\ell^{f(v, w)}\mathbb{Z}
$$
the unique homomorphism induced by sending each copy of $\mathcal{A}$ to $\Z/\ell^{f(v)}\Z \times \Z/\ell^{f(w)}\Z$ followed by the determinant. Then we have that
$$
\psi_\bullet: \wedge^2_{\textup{prof}} \mathcal{A} \to \prod_{v < w} \mathbb{Z}/\ell^{f(v, w)}\mathbb{Z}
$$
is a continuous isomorphism. 
\end{proposition}

\begin{proof}
It suffices to show that the alternating map $b_\bullet: \mathcal{A} \times \mathcal{A} \to \prod_{v < w} \mathbb{Z}/\ell^{f(v, w)}\mathbb{Z}$ corresponding to $\psi_\bullet$ satisfies the universal property of $\wedge^2_{\text{prof}}$. To this end, let 
$$
b: \mathcal{A} \times \mathcal{A} \to \mathcal{B}
$$
be a continuous alternating map, where $\mathcal{B}$ is another profinite abelian group. Let $e_v$ be the vector in $\mathcal{A}$ having all coordinates equal to $0$ except at $v$ where it equals $1$. Likewise, for $v < w$, let $e(v, w)$ be the vector in $\prod_{v < w} \mathbb{Z}/\ell^{f(v, w)}\mathbb{Z}$ having all coordinates equal to $0$ except at $(v, w)$ where it equals $1$. Observe that $\{e_v : v \in \Omega\}$ is a set of topological generators of $\mathcal{A}$, while $\{e(v, w) : v < w\}$ is a set of topological generators of $\prod_{v < w} \mathbb{Z}/\ell^{f(v, w)}\mathbb{Z}$. Further observe that
$$
b_\bullet(e_v, e_w) = e(v, w)
$$
for each $v < w$. Therefore there is at most one continuous homomorphism
$$
\phi: \prod_{v < w} \mathbb{Z}/\ell^{f(v, w)}\mathbb{Z} \to \mathcal{B}
$$
satisfying $b = \phi \circ b_\bullet$, since this equation already determines $\phi$ on a set of topological generators. This shows that $\phi$, if it exists, is unique. We now prove existence. 

Let $U$ be an open subgroup of $\mathcal{B}$. Consider the bilinear, alternating map
$$
\pi_U \circ b: \mathcal{A} \times \mathcal{A} \to \mathcal{B}_U.
$$
We now define
$$
\phi_U((\gamma(v, w))_{v < w}) := \sum_{v < w} \gamma(v, w) \cdot \pi_U(b(e_v, e_w)).
$$
As explained in the proof of Proposition \ref{Universal property}, the bilinear, alternating map $\pi_U \circ b$ has in its left and right kernel all but finitely many of the $e_v$. Therefore the above sum contains only finitely many terms. In particular, $\phi_U$ is continuous. The map $\phi_U$ satisfies
$$
\phi_U \circ b_\bullet = \pi_U \circ b.
$$
Furthermore, $\phi_U$ is unique as we have already shown uniqueness. This forces the system of maps $\phi_U$ to be compatible. Hence they induce a continuous homomorphism
$$
\phi: \prod_{v < w} \mathbb{Z}/\ell^{f(v, w)}\mathbb{Z} \to \mathcal{B}
$$
with the property that $\phi_U \circ b_\bullet = \pi_U \circ b$ for all open subgroups $U$. Since this holds for all open subgroups $U$, we must have that $\phi \circ b_\bullet = b$. Hence we have shown the desired universality, completing the proof of the proposition.
\end{proof}

\subsection{\texorpdfstring{Reconstructing $\mathcal{S}$ from $\Phi_{\mathcal{S}}$}{Reconstructing the group}}
We are now ready to attach a decorated graph to a $2$-nilpotent pro-$\ell$ group on $(\Omega, f, g)$. Throughout this subsection, we always assume that $\ell$ is odd. We also fix a well-ordering $\leq$ on the set $\Omega$ with the convention that $v_\ell(1) < v_\ell(2)$ are the first two minima of $\Omega$. 

\begin{proposition} 
\label{prop:defining the invariant}
Let $\mathcal{S} := (\mathcal{G}, \pi_\bullet)$ be a $2$-nilpotent pro-$\ell$ group on $(\Omega, f, g)$. For each $v < w$ in $\Omega$, there exists exactly one pair $\lambda_\mathcal{S}(v, w), \lambda_\mathcal{S}(w, v)$ in $\mathbb{Z}/\ell^{f(v, w)}\mathbb{Z}$ such that  
$$
\pi_v \cup \pi_w + \lambda_\mathcal{S}(w, v) \cdot \tilde{\theta}_v + \lambda_\mathcal{S}(v, w) \cdot \tilde{\theta}_w = 0 \quad \textup{in } H^2(\mathcal{G}, \mathbb{Z}/\ell^{f(v, w)}\mathbb{Z})
$$
and such that $\lambda_\mathcal{S}(w, v) = 0$ whenever $f(v) = \infty$.

Furthermore, the $4$-tuple $\Phi_\mathcal{S} := (\Omega, f, g, \lambda_\mathcal{S})$ is a decorated graph. 
\end{proposition}

\begin{proof}
We will start with the case where $v, w \in \Omega_{\text{fin}}$. Observe that $\pi_v \cup \pi_w$ is a class in the image of the inflation map
$$
H^2(\mathbb{Z}/\ell^{f(v)}\mathbb{Z} \times \mathbb{Z}/\ell^{f(w)}\mathbb{Z}, \mathbb{Z}/\ell^{f(v, w)}\mathbb{Z}) \xrightarrow{\text{inf}} H^2(\mathcal{G}, \mathbb{Z}/\ell^{f(v, w)}\mathbb{Z}).
$$
Therefore, by property (C3) of Definition \ref{def: Rado groups}, we have that this class is also in the inflation 
$$
\text{Ext}^1_{\mathbb{Z}_\ell}(\mathbb{Z}/\ell^{f(v)}\mathbb{Z} \times \mathbb{Z}/\ell^{f(w)}\mathbb{Z}, \mathbb{Z}/\ell^{f(v, w)}\mathbb{Z}) \xrightarrow{\text{inf}} H^2(\mathcal{G}, \mathbb{Z}/\ell^{f(v, w)}\mathbb{Z}).
$$
This latter inflation is generated by $\{\tilde{\theta}_v, \tilde{\theta}_w\}$ if $v, w \in \Omega_{\text{fin}}$: here $\tilde{\theta}_v$ has been defined in property (C4) of Definition \ref{def: Rado groups}. This gives the existence of the coefficients $\lambda_\mathcal{S}(v, w)$ and $\lambda_\mathcal{S}(w, v)$ satisfying the desired property. To demonstrate uniqueness, we need to prove that the inflation map 
$$
\text{Ext}^1_{\mathbb{Z}_\ell}(\mathbb{Z}/\ell^{f(v)}\mathbb{Z} \times \mathbb{Z}/\ell^{f(w)}\mathbb{Z}, \mathbb{Z}/\ell^{f(v, w)}\mathbb{Z}) \xrightarrow{\text{inf}} H^2(\mathcal{G}, \mathbb{Z}/\ell^{f(v, w)}\mathbb{Z})
$$ 
is injective. Observe that, by the universal property of the abelianization, the inflation map
$$
\text{Ext}^1_{\mathbb{Z}_\ell}(\mathcal{G}^{\text{ab}}, \mathbb{Z}/\ell^{f(v, w)}\mathbb{Z}) \xrightarrow{\text{inf}} H^2(\mathcal{G}, \mathbb{Z}/\ell^{f(v, w)}\mathbb{Z})
$$
is injective. Therefore it suffices to prove that the inflation map
$$
\text{Ext}^1_{\mathbb{Z}_\ell}(\mathbb{Z}/\ell^{f(v)}\mathbb{Z} \times \mathbb{Z}/\ell^{f(w)}\mathbb{Z}, \mathbb{Z}/\ell^{f(v, w)}\mathbb{Z}) \xrightarrow{\text{inf}} \text{Ext}^1_{\mathbb{Z}_\ell}(\mathcal{G}^{\text{ab}}, \mathbb{Z}/\ell^{f(v, w)}\mathbb{Z})
$$
is injective. To this end, suppose that we have a finite abelian group $G$ and a surjection
$$
G \twoheadrightarrow \mathbb{Z}/\ell^{f(v)}\mathbb{Z} \times \mathbb{Z}/\ell^{f(w)}\mathbb{Z}
$$
that becomes split after pulling it back via $(\pi_v, \pi_w): \mathcal{G}^{\text{ab}} \rightarrow \mathbb{Z}/\ell^{f(v)}\mathbb{Z} \times \mathbb{Z}/\ell^{f(w)} \mathbb{Z}$. This is equivalent to the existence of a lift of $(\pi_v, \pi_w): \mathcal{G}^{\text{ab}} \rightarrow \mathbb{Z}/\ell^{f(v)}\mathbb{Z} \times \mathbb{Z}/\ell^{f(w)}\mathbb{Z}$ to a continuous homomorphism $f$ from $\mathcal{G}^{\text{ab}}$ to $G$. Then we must show that the surjection $G \twoheadrightarrow \mathbb{Z}/\ell^{f(v)}\mathbb{Z} \times \mathbb{Z}/\ell^{f(w)}\mathbb{Z}$ itself must split. We will now use property (C2) of Definition \ref{def: Rado groups}. Indeed, the image of $f$ on the subgroup consisting of elements that after having applied $\pi_\bullet$ have all coordinates outside of $v$, $w$ equal to $0$ gives such a splitting. 

The other cases are very similar, with the only difference that now the image of the inflation is generated by the smaller set $\{\tilde{\theta}_v : f(v) \neq \infty\} \cup \{\tilde{\theta}_w : f(w) \neq \infty\}$. The rest of the argument is completely analogous.

It remains to show that $(\Omega, f, g, \lambda_{\mathcal{S}})$ is a decorated graph. Property (C4) of Definition \ref{def: Rado groups} makes sure that the conditions on $\lambda_\mathcal{S}(v_\ell(1), -)$ are satisfied. Our own conventions on $\lambda_\mathcal{S}(-, v_\ell(i))$ ensure that the conditions on these labels are also met. 
\end{proof}

Our end goal in this section is to reconstruct $\mathcal{S}$ from the graph $\Phi_{\mathcal{S}}$. 

\begin{proposition} 
\label{free on abelianization}
Let $\mathcal{S} := (\mathcal{G}, \pi_\bullet)$ be a $2$-nilpotent pro-$\ell$ group on $(\Omega, f, g)$. Then $\mathcal{G}$ is free on its abelianization. 
\end{proposition}

\begin{proof}
Write $\text{red}_{\ell^n}: \Z_\ell \rightarrow \Z/\ell^n\Z$ for the reduction map. Thanks to Subsection \ref{intermezzo}, we have that the image of the inflation map
$$
H^2(\mathcal{G}^{\text{ab}}, \Q/\Z) \xrightarrow{\text{inf}} H^2(\mathcal{G}, \Q/\Z)
$$
is generated by the set 
$$
\{\pi_v \cup \pi_w : v, w \in \Omega, \{v, w\} \neq \{v_\ell(1), v_\ell(2)\}\} \cup \{(\text{red}_{\ell^n} \circ \pi_{v_\ell(1)}) \cup (\text{red}_{\ell^n} \circ \pi_{v_\ell(2)}) : n \in \Z_{\geq 0}\}
$$
viewed as a subset of $H^2(\mathcal{G}, \Q/\Z)$ via the natural inclusion $\mathbb{Z}/\ell^{f(v, w)}\mathbb{Z} \subseteq \Q/\Z$ sending $1$ to $\frac{1}{\ell^{f(v, w)}}$. Therefore we deduce from property (C3) of Definition \ref{def: Rado groups} that each of these generators vanishes when viewed with coefficients in $\Q/\Z$ as the abelian extension classes $\tilde{\theta}_v$ have, trivially, vanishing image in $H^2(\mathcal{G}^{\text{ab}}, \Q/\Z)$.
% Need to interpret pi_v cup pi_w accordingly for f(v) = f(w) = infty
\end{proof}

With the work done so far, we have already determined the structure of $[\mathcal{G}, \mathcal{G}]$: indeed, combining Proposition \ref{inflation of h^2=0}, Proposition \ref{computation of profinite wedge 2}, Proposition \ref{free on abelianization} and property (C2) of Definition \ref{def: Rado groups}, we obtain an isomorphism, natural in $\mathcal{S}$, with $\prod_{v < w} \mathbb{Z}/\ell^{f(v, w)}\mathbb{Z}$.

Now that we understand ${\mathcal{G}}^{{\text{ab}}}$ and $[\mathcal{G}, \mathcal{G}]$, we want to describe the extension class: as we are going to show, that information is precisely encoded in the decorated graph $\Phi_\mathcal{S}$. As a first step towards this goal, let us attach to each decorated graph $\Phi$ a $2$-nilpotent pro-$\ell$ group on $(\Omega, f, g)$, which we will denote by $\text{Gr}(\Phi)$. This is done as follows. 

We start by defining the set
$$
\mathcal{G}(\Phi) := \left(\prod_{v < w} \mathbb{Z}/\ell^{f(v, w)}\mathbb{Z} \cdot \{(v, w)\}\right) \times \prod_{v \in \Omega}\mathbb{Z}/\ell^{f(v)}\mathbb{Z}.
$$
We next equip $\mathcal{G}(\Phi)$ with a group law. Fix for each $v \in \Omega$ a $2$-cocycle $(\prod_{v \in \Omega}\mathbb{Z}/\ell^{f(v)}\mathbb{Z})^2$ to $\mathbb{Z}/\ell^{f(v)}\mathbb{Z}$ representing the inflation of $\theta_v$ to $\prod_{v \in \Omega}\mathbb{Z}/\ell^{f(v)}\mathbb{Z}$ and vanishing on $(0,0)$. We will abuse notation by continuing to write $\theta_v$ for its inflation and for the natural reduction of its inflation modulo $\ell^m$ for $m \leq f(v)$. 

We view $\pi_v \cup \pi_w$ as a $2$-cocycle from $(\prod_{v \in \Omega}\mathbb{Z}/\ell^{f(v)}\mathbb{Z})^2$ to $\mathbb{Z}/\ell^{f(v, w)}\mathbb{Z}$, with the assignment that sends a pair of elements $(g_1, g_2)$ to $\pi_v(g_1) \cdot \pi_w(g_2)$, where the product takes place in $\mathbb{Z}/\ell^{f(v, w)}\mathbb{Z}$. We can now package them together in one $2$-cocycle $\theta_\Phi$ whose $(v, w)$-coordinate is $\pi_v \cup \pi_w + \lambda(w, v) \cdot \theta_v + \lambda(v, w) \cdot \theta_w$ 
$$
\theta_\Phi: \prod_{v \in \Omega} (\mathbb{Z}/\ell^{f(v)}\mathbb{Z})^2 \to \prod_{v < w} \mathbb{Z}/\ell^{f(v, w)}\mathbb{Z}. 
$$
We now define the group law $*_\Phi$ on $\mathcal{G}(\Phi)$ via the formula
\begin{align}
\label{eCocycleGroup}
(\mu_\bullet, \rho_\bullet) *_\Phi (\mu_\bullet', \rho_\bullet') = (\mu_\bullet + \mu_\bullet' + \theta_\Phi(\rho_\bullet, \rho_\bullet'), \rho_\bullet + \rho_\bullet').
\end{align}
This is a group law with the neutral element being $(0, 0)$. Using that each coordinate is a $2$-cocycle, $\theta_\Phi$ defines a central extension of $\prod_{v \in \Omega} \mathbb{Z}/\ell^{f(v)}\mathbb{Z}$. Of course we have a natural surjective projection homomorphism
$$
\pi(\Phi): \mathcal{G}(\Phi) \to \prod_{v \in \Omega}\mathbb{Z}/\ell^{f(v)}\mathbb{Z}.
$$
We now put
$$
\text{Gr}(\Phi) := (\mathcal{G}(\Phi), \pi(\Phi)).
$$

\begin{proposition} 
\label{it's acceptable group}
The pair $\textup{Gr}(\Phi)$ defines a $2$-nilpotent pro-$\ell$ group on $(\Omega, f, g)$.
\end{proposition}

\begin{proof}
Let us verify the properties one by one. 

Certainly, $\mathcal{G}(\Phi)$ is a pro-$\ell$ group having an epimorphism to a pro-$\ell$ group with kernel a pro-$\ell$ group. More precisely, $\mathcal{G}(\Phi)$ is a central extension of an abelian pro-$\ell$ group, and therefore it has class at most $2$. Take distinct $v, w \in \Omega$. Let $\sigma \in \mathcal{G}(\Phi)$ be an element that projects under $\pi(\Phi)$ to the unique element that is $0$ outside of $v$ and $1$ on the $v$ coordinate. Define $\tau \in \mathcal{G}(\Phi)$ similarly with $w$ taking the role of $v$. From the cocycle representation, one gets that the commutator $[\sigma, \tau]$ projects to $0$ outside of $\{v, w\}$, and is a topological generator of the $\{v, w\}$ coordinate. Hence $\mathcal{G}(\Phi)$ is non-abelian, and so its class is precisely $2$ as desired. This establishes (C1).

Furthermore, the last calculation shows that the commutators topologically generate a group containing $\ker(\pi(\Phi))$ and therefore this subgroup must coincide with the commutator subgroup, establishing (C2).

Next we observe that, thanks to Subsection \ref{intermezzo},
$$
H^2(\mathbb{Z}/\ell^{f(v)}\mathbb{Z} \times \mathbb{Z}/\ell^{f(w)}\mathbb{Z}, \mathbb{Z}/\ell^{f(v, w)}\mathbb{Z})
$$ 
is generated by $\pi_v \cup \pi_w, \theta_v, \theta_w$. On the other hand, we have by construction that
$$
\pi_v \cup \pi_w + \lambda(w, v) \cdot \theta_v + \lambda(v, w) \cdot \theta_w
$$
vanishes when inflated to $\mathcal{G}(\Phi)$. Indeed, let $\phi(v, w)$ be the $1$-cochain given by the set-theoretic projection of $\mathcal{G}(\Phi)$ on the $\{v, w\}$-coordinate. Then
$$
d\phi(v, w) = \pi_v \cup \pi_w + \lambda(w, v) \cdot \theta_v + \lambda(v, w) \cdot \theta_w,
$$
as desired. So this shows that the inflation of $H^2(\mathbb{Z}/\ell^{f(v)}\mathbb{Z} \times \mathbb{Z}/\ell^{f(w)}\mathbb{Z}, \mathbb{Z}/\ell^{f(v, w)}\mathbb{Z})$ coincides with the inflation of $\text{Ext}^1_{\mathbb{Z}_\ell}(\mathbb{Z}/\ell^{f(v)}\mathbb{Z} \times \mathbb{Z}/\ell^{f(w)}\mathbb{Z}, \mathbb{Z}/\ell^{f(v, w)}\mathbb{Z})$, giving (C3).

The assumption that $\Phi$ is a decorated graph deals automatically with condition (C4).
\end{proof}

\begin{theorem} 
\label{Thm: the graph determines the group}
Let $\mathcal{S} := (\mathcal{G}, \pi_\bullet)$ be a $2$-nilpotent pro-$\ell$ group on $(\Omega, f, g)$. Then $\mathcal{S}$ is isomorphic to $\textup{Gr}(\Phi_\mathcal{S})$ as $2$-nilpotent pro-$\ell$ groups on $(\Omega, f, g)$. 
\end{theorem}

\begin{proof}
Thanks to Proposition \ref{prop:defining the invariant}, there exists an $1$-cochain $\phi(v, w): \mathcal{G} \to \mathbb{Z}/\ell^{f(v, w)}\mathbb{Z}$ for each $v < w$ such that
$$
d\phi(v, w) = \pi_v \cup \pi_w + \lambda(w, v) \cdot \theta_v + \lambda(v, w) \cdot \theta_w.
$$
This can be packaged into a homomorphism
$$
(\phi_\bullet, \pi_\bullet): \mathcal{G} \to \mathcal{G}(\Phi_\mathcal{S}).
$$
Since $\pi_\bullet, \pi(\Phi_\mathcal{S})$ are precisely the abelianization maps, we see that this homomorphism induces an isomorphism on the abelianizations. Therefore, by Proposition \ref{Enough surjective on abelianization}, it suffices to prove that both groups are free on their abelianization. Thanks to Proposition \ref{free on abelianization} this holds once we know that both $\mathcal{S}$ and $\text{Gr}(\Phi_\mathcal{S})$ are $2$-nilpotent pro-$\ell$ groups on $(\Omega, f, g)$. But for $\mathcal{S}$ this is true by assumption, while for $\text{Gr}(\Phi_\mathcal{S})$ this has been shown in Proposition \ref{it's acceptable group}. 
\end{proof}

We now show that is sufficient to control the isomorphism class of the graph to control the isomorphism class of the group. As in the rest of this section, we assume that each triple $(\Omega, f, g)$ comes with a well-ordering having $v_\ell(1), v_\ell(2)$ as the first two minima. 

\begin{proposition} 
\label{isomorphic graphs, isomorphic groups}
Suppose that $\Phi_1: = (\Omega_1, f_1, g_1, \lambda_1)$ and $\Phi_2: = (\Omega_2, f_2, g_2, \lambda_2)$ are isomorphic as decorated graphs. Then the pairs $\textup{Gr}(\Phi_1)$ and $\textup{Gr}(\Phi_2)$ are isomorphic.
\end{proposition}
    
\begin{proof}
Let $\gamma: \Omega_1 \to \Omega_2$ be the decorated graph isomorphism. Since $\gamma$ needs to preserve $f$, we find a well-defined set-theoretic bijection $\phi: \mathcal{G}(\Phi_1) \to \mathcal{G}(\Phi_2)$ given by the formula
$$
\phi((\mu_{v_1, w_1})_{v_1<w_1 \in \Omega_1}, (\rho_{v_1})_{v_1 \in \Omega_1}) = ((\mu_{\gamma^{-1}(v_2), \gamma^{-1}(w_2)})_{\gamma^{-1}(v_2)<\gamma^{-1}(w_2) \in \Omega_2}, (\rho_{\gamma^{-1}(v_2)})_{v_2 \in \Omega_2}).
$$
Since $\gamma$ preserves $g$, we find that $\theta_v = \theta_{\gamma(v)} \circ \phi$ for each $v \in \Omega_{1, \text{fin}}$. Finally, using that $\gamma$ preserves $\lambda$, we see that 
$$
\theta_{\Phi_1} = \theta_{\Phi_2} \circ \phi.
$$
From the shape of the group law, we deduce that $\phi$ is also a group homomorphism. In total, we have that $(\phi, \gamma)$ is an isomorphism between $\text{Gr}(\Phi_1)$ and $\text{Gr}(\Phi_2)$. 
\end{proof}

We now define \emph{weakly Rado groups}.

\begin{definition}
Let $\mathcal{S}$ be a $2$-nilpotent pro-$\ell$ group on $(\Omega, f, g)$. We say that it is \emph{weakly Rado} in case the graph $\Phi_{\mathcal{S}}$ is weakly Rado. 
\end{definition}

Let $\mathcal{S} := (\mathcal{G}, \pi_\bullet)$ be a $2$-nilpotent pro-$\ell$ group on $(\Omega, f, g)$ and let $\gamma_\bullet \in \prod_{v \in \Omega_{\text{fin}}}(\mathbb{Z}/\ell^{f(v)}\mathbb{Z})^\ast$. Then we can diagonally multiply $\pi_\bullet$ by $\gamma_\bullet$ in the $\Omega_{\text{fin}}$-coordinates: we denote the result as $\gamma_\bullet \cdot \pi_\bullet$. Let us verify that the pair
$$
\gamma_\bullet\cdot \mathcal{S} := (\mathcal{G}, \gamma_\bullet \cdot \pi_\bullet)
$$
is still a $2$-nilpotent pro-$\ell$ group on $(\Omega, f, g)$. 

Condition (C1) of Definition \ref{def: Rado groups} is for free as we did not change the group $\mathcal{G}$. Since the homomorphisms $\pi_\bullet$ and $\gamma_\bullet \cdot \pi_\bullet$ differ by an automorphism of the target, this does not change the kernel. This gives (C2). Next, since the automorphism is diagonal, the images of the two inflations in (C3) are unaffected. Finally, (C4) comes for free from the fact that the action did not involve the $v_\ell(1)$ or $v_\ell(2)$ coordinate.

We also have the relation
\begin{align}
\label{eLambdaTransform}
\lambda_{\gamma_\bullet \cdot \mathcal{S}} = \gamma_\bullet \cdot \lambda_\mathcal{S}
\end{align}
between the resulting labels. By Theorem \ref{prob=1}, we may certainly fix one Rado graph $\Phi(\text{Rado})$. Define $\text{Gr}(\Phi(\text{Rado})) := (\mathcal{G}(\text{Rado}), \pi_\bullet(\text{Rado}))$. The following theorem is the key result for our application. 

\begin{theorem} 
\label{weakly Rado groups are isomorphic}
Let $(\mathcal{G}, \pi), (\mathcal{G}', \pi')$ be $2$-nilpotent pro-$\ell$ groups on $(\Omega, f, g)$ and $(\Omega', f', g')$ respectively. Suppose that they are both weakly Rado. Then
$$
\mathcal{G} \simeq_{\textup{top.gr.}} \mathcal{G}'.
$$
More precisely, $\mathcal{G}$ and $\mathcal{G}'$ are isomorphic to $\mathcal{G}(\textup{Rado})$.
\end{theorem}

\begin{proof}
Thanks to Proposition \ref{rado up to scaling if and only if weakly rado} and equation (\ref{eLambdaTransform}), we find that there exists a scaling $\gamma_\bullet$ such that $\Phi_{\gamma_\bullet \cdot \mathcal{S}}$ is Rado. By Theorem \ref{Thm: the graph determines the group}, it follows that $\gamma_\bullet \cdot \mathcal{S}$ is isomorphic to $\text{Gr}(\Phi_{\gamma_\bullet \cdot \mathcal{S}})$ as $2$-nilpotent pro-$\ell$ group on $(\Omega, f, g)$. Thanks to Theorem \ref{Thm:two Rado graphs are isomorphic} and Proposition \ref{isomorphic graphs, isomorphic groups}, we know that this latter group is isomorphic to $\text{Gr}(\Phi(\text{Rado}))$ as $2$-nilpotent pro-$\ell$ groups on $(\Omega, f, g)$. In total, since the diagonal scaling leaves the groups unaffected, we find out that both $\mathcal{G}$ and $\mathcal{G}'$ are isomorphic, as topological groups, to $\mathcal{G}(\text{Rado})$.
\end{proof}

We end this section with a reconstruction theorem. In Remark \ref{rmk: for Rado graphs the cyclotomic places is reconstructed} we saw that one may distinguish between the two labels $\{v_\ell(1), v_\ell(2)\}$ in the case of Rado graphs. We will now show how to reconstruct the character $\pi_{v_\ell(1)}$ up to scaling for weakly Rado groups.

\begin{proposition} 
\label{reconstructing cyclotomic abstract}
Let $\mathcal{S} := (\mathcal{G}, \pi_{\bullet})$ be a weakly Rado group on $(\Omega, f, g)$. Then the subgroup $\textup{ker}(\pi_{v_\ell(1)})$ is stable under $\textup{Aut}_{\textup{top.gr.}}(\mathcal{G})$.
\end{proposition}

\begin{proof}
Let $\pi$ be an element of the $\textup{Aut}_{\textup{top.gr.}}(\mathcal{G})$-orbit of $\pi_{v_{\ell}(1)}$. Thanks to part (C2) of Definition \ref{def: Rado groups}, we have $\lambda, \mu \in \Z_{\ell}$ such that
$$
\pi = \lambda \cdot \pi_{v_{\ell}(1)} + \mu \cdot \pi_{v_{\ell}(2)}.
$$
The desired conclusion is equivalent to the claim $\mu = 0$. Let us pinpoint what in the end will be a characterizing property for the $\Z_{\ell}$-character $\pi_{v_{\ell}(1)}$. Take a positive integer $n$ and take a continuous homomorphism $\pi_0: \mathcal{G} \to \Z/\ell^n\Z$. Then
$$
\ell \cdot \pi_{v_{\ell}(1)} \cup \pi_0 = 0
$$
in $H^2(\mathcal{G}, \Z/\ell^n\Z)$. Indeed, this follows upon combining property (C4) of Definition \ref{def: Rado groups}, which gives this conclusion for the characters $\pi_v$ with $v \in \Omega_{\text{fin}}$, and property (C2).

This property is clearly preserved by the action of $\textup{Aut}_{\textup{top.gr.}}(\mathcal{G})$ on $\Z_{\ell}$-characters. It follows therefore that $\mu \cdot \pi_{v_{\ell}(2)}$ has this property. If $\mu$ were not to be $0$, we readily reach a contradiction upon combining Proposition \ref{prop:defining the invariant} with the extension property of a weakly Rado graph. 
\end{proof}

\section{\texorpdfstring{Rado groups for $\ell = 2$}{Rado groups for 2}}
\label{Rado groups at 2}
We now turn to $\ell = 2$.  Let $(\Omega, f, g)$ be a triple as in Section \ref{decorated graphs at 2} satisfying the axioms of a decorated graph involving only these three variables. 

\begin{definition} 
\label{def: Rado groups at 2} 
A $2$-nilpotent pro-$2$ group on $(\Omega, f, g)$ is a pair $\mathcal{S} := (\mathcal{G}, \pi_\bullet)$ as in Definition \ref{def: Rado groups}, except that condition (C4) is slightly weakened as follows:
\begin{enumerate}
\item[(C4)] we attach to each $v \in \Omega_{\textup{fin}}$ the unique extension class, which we denote by $\theta_v$, in $\textup{Ext}^1_{\mathbb{Z}_2}(\mathbb{Z}/2^{f(v)}\mathbb{Z}, \mathbb{Z}/2^{f(v)}\mathbb{Z})$ corresponding to the exact sequence where lifting $1$ and multiplying it by $2^{f(v)} \cdot g(v)$ yields $1$. Let us call $\tilde{\theta}_v$ the inflation of this class to $H^2(\mathcal{G}, \mathbb{Z}/2^{f(v)}\mathbb{Z})$ through $\pi_v$. Then in case $f(v) \geq 2$, we have that 
$$
\pi_{v_2(1)} \cup \pi_v + \log_2(1+ 2^{f(v)} \cdot g(v)) \cdot \tilde{\theta}_v
$$
equals $0$ in $H^2(\mathcal{G}, \mathbb{Z}/2^{f(v)}\mathbb{Z})$. Furthermore, we have that
$$
(\textup{red}_2 \circ \pi_{v_2(2)}) \cup (\textup{red}_2 \circ \pi_v) \textup{ is trivial in } H^2(\mathcal{G}, \Z/2\Z) \Longleftrightarrow f(v) \geq 2,
$$
where $\textup{red}_2: \Z/2^k\Z \rightarrow \Z/2\Z$ is the natural quotient map.
\end{enumerate}
\end{definition}

Let $\mathcal{S}$, $\mathcal{S}'$ be a pair of $2$-nilpotent pro-$2$ groups on respectively $(\Omega, f, g)$ and $(\Omega', f', g')$. We define an isomorphim between $\mathcal{S}$ and $\mathcal{S}'$ exactly as in Section \ref{Rado group odd primes}. Our next goal is to attach to each $\mathcal{S}$ a certain labeling $\lambda_\mathcal{S}$, which will turn 
$$
\Phi_\mathcal{S} := (\Omega, f, g, \lambda_\mathcal{S})
$$
into a decorated graph. The only difference with the case $\ell$ odd is in the slightly different constraints on the labels between the fiber above $\infty$ and the fibers above $\mathbb{Z}_{\geq 1}$. But the basic mechanics of the proof are all based on Subsection \ref{intermezzo}, where $\ell = 2$ is conveniently allowed. Like in Section \ref{Rado group odd primes}, we fix a well-ordering $\leq$ on $\Omega$ with the requirement that $v_2(1) < v_2(2)$ are the first two minima.  

With this in mind, a straightforward adaptation of the proofs in Section \ref{Rado group odd primes} yields the following analogous facts. 

\begin{proposition} 
\label{prop: defining the invariant at 2}
Let $\mathcal{S} := (\mathcal{G}, \pi_\bullet)$ be a $2$-nilpotent pro-$2$ group on $(\Omega, f, g)$. For each $v < w$ in $\Omega$, there exists exactly one pair $\lambda_\mathcal{S}(v, w), \lambda_\mathcal{S}(w, v)$ in $\mathbb{Z}/2^{f(v, w)}\mathbb{Z}$ such that  
$$
\pi_v \cup \pi_w + \lambda_\mathcal{S}(w, v) \cdot \tilde{\theta}_v + \lambda_\mathcal{S}(v, w) \cdot \tilde{\theta}_w = 0 \quad \textup{in } H^2(\mathcal{G}, \mathbb{Z}/2^{f(v, w)}\mathbb{Z})
$$
and such that $\lambda_\mathcal{S}(w, v) = 0$ whenever $f(v) = \infty$.

Furthermore, the $4$-tuple $\Phi_\mathcal{S} := (\Omega, f, g, \lambda_\mathcal{S})$ is a decorated graph.    
\end{proposition}

\begin{proposition}
\label{free on the abelianization at 2}
Let $\mathcal{S} := (\mathcal{G}, \pi)$ be a $2$-nilpotent pro-$2$ group on $(\Omega, f, g)$. Then $\mathcal{G}$ is free on its abelianization.     
\end{proposition}

Following the recipe given in Section \ref{Rado group odd primes}, we can attach a $2$-nilpotent pro-$2$ group to a decorated graph $\Phi$, which we denote by $\text{Gr}(\Phi)$.

\begin{theorem}  
\label{Thm: the graph determines the group at 2}
Let $\mathcal{S}$ be a $2$-nilpotent pro-$2$ group on $(\Omega, f, g)$. Then $\mathcal{S}$ is isomorphic to $\textup{Gr}(\Phi_\mathcal{S})$ as $2$-nilpotent pro-$2$ groups on $(\Omega, f, g)$. 
\end{theorem}

Next, following the proof of Proposition \ref{isomorphic graphs, isomorphic groups}, we see that any isomorphism between the graphs translates into an isomorphism of groups.

\begin{proposition} 
\label{isomorphic graphs, isomorphic groups at 2}
Suppose that $\Phi_1: = (\Omega_1, f_1, g_1, \lambda_1)$ and $\Phi_2: = (\Omega_2, f_2, g_2, \lambda_2)$ are isomorphic as decorated graphs. Then the pairs $\textup{Gr}(\Phi_1)$ and $\textup{Gr}(\Phi_2)$ are isomorphic.
\end{proposition}

We can now pinpoint the groups needed for our application. 

\begin{definition} 
\label{def: reciprocity groups}
Let $\mathcal{S}$ be a $2$-nilpotent pro-$2$ group on $(\Omega, f, g)$. We say that $\mathcal{S}$ is a \emph{reciprocity} $2$-nilpotent pro-$2$ group on $(\Omega, f, g)$ precisely when $\Phi_\mathcal{S}$ is a decorated reciprocity graph.    
\end{definition}

\begin{definition} 
\label{def: weakly rado groups at 2}
Let $\mathcal{S}$ be a reciprocity $2$-nilpotent pro-$2$ group on $(\Omega, f, g)$. We say that it is \emph{weakly Rado} in case the graph $\Phi_\mathcal{S}$ is weakly Rado. 
\end{definition}

Following the argument given in Section \ref{Rado group odd primes} for Theorem \ref{weakly Rado groups are isomorphic}, we obtain the main result of this section.

\begin{theorem} 
\label{weakly Rado groups are isomorphic at 2}
Let $(\mathcal{G}, \pi), (\mathcal{G}', \pi')$ be reciprocity $2$-nilpotent pro-$2$ groups on $(\Omega, f, g)$ and $(\Omega', f', g')$ respectively. Suppose that they are both weakly Rado. Then
$$
\mathcal{G} \simeq_{\textup{top.gr.}} \mathcal{G}'.
$$
\end{theorem}

We end the section with the analogue of Proposition \ref{reconstructing cyclotomic abstract} for $\ell = 2$.

\begin{proposition} 
\label{reconstructing cyclotomic abstract at 2}
Let $\mathcal{S} := (\mathcal{G}, \pi_{\bullet})$ be a weakly Rado group on $(\Omega, f, g)$. Then the subgroup $\textup{ker}(\pi_{v_2(1)})$ is stable under $\textup{Aut}_{\textup{top.gr.}}(\mathcal{G})$.
\end{proposition}

\begin{proof}
The proof is identical to the proof of Proposition \ref{reconstructing cyclotomic abstract}.
\end{proof}

\section{Galois groups} 
\label{section: Galois groups l odd}
Let $\ell$ be an odd prime number. Fix an imaginary quadratic number field $K$. Denote by $\mathcal{O}_K$ its ring of integers and by $\text{Cl}(K)$ its class group. For the rest of this section we will operate under the assumption that $\text{Cl}(K)[\ell] = \{0\}$. Furthermore, in case $\ell = 3$, we assume that $K \otimes_{\mathbb{Q}} \mathbb{Q}_3$ does not contain a primitive third root of unity. 

In the rest of this section, $v_\ell(1)$ denotes the cyclotomic $\mathbb{Z}_\ell$-extension of $K$ and $v_\ell(2)$ denotes the anti-cyclotomic extension of $K$.

Furthermore, define $\Omega := \{v_\ell(1), v_\ell(2)\} \cup \Omega_{\text{fin}}$, where $\Omega_{\text{fin}}$ consists of the places $v \in \Omega_K$ such that $\mathcal{O}_K/v$ has a primitive $\ell$-th root of unity (i.e. the size of $\mathcal{O}_K/v$ is congruent to $1$ modulo $\ell$). We define $f(v)$ to be the $\ell$-adic valuation of $\#(\mathcal{O}_K/v)^\ast$ for $v \in \Omega_{\text{fin}}$ and we define $f(v) := \infty$ for $v \in \{v_\ell(1), v_\ell(2)\}$. 

\begin{proposition} 
\label{a totally ramified extension}
For each $v \in \Omega_{\textup{fin}}$, there exists a character $\pi_v: \mathcal{G}_K \twoheadrightarrow \mathbb{Z}/\ell^{f(v)}\mathbb{Z}$ that is totally ramified at $v$ and unramified elsewhere.
\end{proposition}

\begin{proof}
Let $v$ be as in the statement. Recall that in general we have an exact sequence
$$
0 \to \frac{(\mathcal{O}_K/v)^\ast}{\mathcal{O}_K^\ast} \to \text{Cl}(K, v) \to \text{Cl}(K) \to 0.
$$
Here $\text{Cl}(K, v)$ denotes the ray class group of modulus $v$. When we tensor this sequence with $\mathbb{Z}_\ell$, the sequence remains exact and the last term vanishes precisely owing to our assumption $\text{Cl}(K)[\ell] = \{0\}$. We remark that our imaginary quadratic field $K$ does not possess a primitive $\ell$-th root of unity. Therefore, recalling once more that $\ell$ is odd and $K$ is imaginary, we have
$$
\mathcal{O}_K^\ast \otimes_{\mathbb{Z}} \mathbb{Z}_\ell=\{0\}.
$$
Overall, the sequence tensored with $\mathbb{Z}_\ell$ yields an identification
$$
(\mathcal{O}_K/v)^\ast[\ell^{\infty}] = \text{Cl}(K, v)[\ell^{\infty}].
$$
The left hand side is a cyclic group of order $\ell^{f(v)}$. Fixing any isomorphism $\pi_v$ between this group and $\mathbb{Z}/\ell^{f(v)}\mathbb{Z}$ gives the desired character.
\end{proof}

For each $v \in \Omega_{\text{fin}}$, fix a choice of $\pi_v$ as guaranteed by Proposition \ref{a totally ramified extension}. Let
$$
\pi_{v_\ell(1)} := \log_\ell \circ \pi_{\text{cyc}}: \mathcal{G}_K \to \mathbb{Z}_\ell,
$$
where $\pi_{\text{cyc}}$ is the cyclotomic character from $\mathcal{G}_K$ to $1+ \ell\cdot \mathbb{Z}_\ell$. Since $K$ is quadratic and $\ell$ is odd, $K$ is linearly disjoint from the unique $\mathbb{Z}_\ell$-extension of $\mathbb{Q}$. Therefore we deduce that $\pi_{v_\ell(1)}$ is \emph{surjective}. Finally, fix any realization $\pi_{v_\ell(2)}: \mathcal{G}_K \twoheadrightarrow \mathbb{Z}_\ell$ of the anti-cyclotomic character. We can now package this construction into a single continuous homomorphism
$$
\pi_\bullet: \mathcal{G}_K \to \prod_{v \in \Omega} \mathbb{Z}/\ell^{f(v)}\mathbb{Z}.
$$

\begin{proposition} 
\label{describing abelianization}
The homomorphism $\pi_\bullet$ is surjective. Furthermore, $\pi_\bullet$ factors through $\mathcal{G}_K(\ell)$ and induces an isomorphism between $\mathcal{G}_K(\ell)/[\mathcal{G}_K(\ell), \mathcal{G}_K(\ell)]$ and $\prod_{v \in \Omega} \mathbb{Z}/\ell^{f(v)}\mathbb{Z}$.
\end{proposition}

\begin{proof}
It suffices to prove the last part of the proposition. It is clear that $\pi_\bullet$ factors through $\mathcal{G}_K(\ell)$ and that $[\mathcal{G}_K(\ell), \mathcal{G}_K(\ell)]$ is inside the kernel of $\pi_\bullet$. We must now check that the induced map between $\mathcal{G}_K(\ell)/[\mathcal{G}_K(\ell), \mathcal{G}_K(\ell)]$ and $\prod_{v \in \Omega} \mathbb{Z}/\ell^{f(v)}\mathbb{Z}$ is an isomorphism. By Pontryagin duality, it suffices to show that the dual map is an isomorphism.

Let us first check injectivity. Injectivity of the dual map is formally equivalent to the following statement: if we have a finite subset $S \subseteq \Omega$ and a $\Z$-linear dependency
\begin{align}
\label{eLinear}
\sum_{v \in S} \gamma_v \cdot \pi_v = 0,
\end{align}
then $\gamma_v$ is divisible by $\ell^{f(v)}$ for each $v \in S$ (so in particular $\gamma_v = 0$ if $f(v) = \infty$). 

Take such a linear dependency. First suppose that there exists an element $v \in \Omega_{\text{fin}} \cap S$. If we restrict equation (\ref{eLinear}) to an inertia subgroup of $v$ in $\mathcal{G}_K$, all terms disappear except for $\gamma_v \cdot \pi_v$. Since the extension corresponding to $\pi_v$ is totally ramified at $v$, the restriction of $\pi_v$ to an inertia subgroup at $v$ still surjects onto $\mathbb{Z}/\ell^{f(v)}\mathbb{Z}$. This forces
$$
\gamma_v \equiv 0 \bmod \ell^{f(v)}
$$
and therefore $\gamma_v \cdot \pi_v = 0$. Hence it suffices to treat the case $S \subseteq \{v_\ell(1), v_\ell(2)\}$. But the cyclotomic and anti-cyclotomic characters are linearly independent over $\mathbb{Z}_\ell$, also ruling out this case.

Let us now check surjectivity. To this end, suppose that there exists a character $\chi$ of $\mathcal{G}_K(\ell)$ that is not in the image of the dual map. Without loss of generality we may assume that the number of ramified primes in the fixed field of $\chi$ is minimal amongst all such characters. If $\chi$ were to ramify at some place $v$, then $v$ must be finite since $\ell$ is odd. If $v$ does not lie above $\ell$, then we may subtract a suitable multiple of $\pi_v$ to make the ramification locus of $\chi$ strictly smaller. If $\chi$ ramifies at some place above $\ell$, then we may subtract multiples of $\pi_{v_\ell(1)}$ and $\pi_{v_\ell(2)}$ to make $\chi$ unramified at all places above $\ell$, i.e. $v_\ell(1)$ and $v_\ell(2)$. Therefore $\chi$ must be unramified. This forces $\chi = 0$ because $\text{Cl}(K)[\ell] = \{0\}$ by assumption, which is the desired contradiction.
\end{proof}

For each $w \in \Omega_{\text{fin}}$, let us put
$$
q(w) := \# \mathcal{O}_K/w.
$$
We write $g(w)$ for the unique element of $(\mathbb{Z}/\ell^{f(w)}\mathbb{Z})^\ast$ satisfying
$$
q(w) \equiv 1+ \ell^{f(w)} \cdot g(w) \bmod \ell^{2f(w)}.
$$

\begin{proposition}
\label{Galois groups are ok}
The pair $\mathcal{S}_K := (\mathcal{G}_K^2(\ell), \pi_\bullet)$ is a $2$-nilpotent pro-$\ell$ group on $(\Omega, f, g)$.  
\end{proposition}

\begin{proof}
Let us verify the desired properties in the corresponding order. Property (C1) holds since $\mathcal{G}_K^2(\ell)$ is by construction a $2$-nilpotent pro-$\ell$ group. Property (C2) is exactly the content of Proposition \ref{describing abelianization}.

Let us now verify (C3). We divide this verification in three steps. In the first step we examine the behavior of $\pi_v \cup \pi_w$ locally at the various places of $K$. In the second step we examine the local behavior of the classes $\tilde{\theta}_v$. In the third step we finish the argument.

\emph{Step 1:} Observe that $\pi_v \cup \pi_w$ is locally trivial at all places different from $v, w$ and places above $\ell$. Indeed, for all such places this is an unramified class, and hence comes from $H^2(\hat{\mathbb{Z}}, \mathbb{Z}/\ell^{f(v, w)}\mathbb{Z}) = 0$. We now claim that the class is trivial also when restricted locally at a prime $\mathfrak{l}$ above $\ell$. Indeed, by our assumption, we know that there are no $\ell$-th roots of unity locally at $\mathfrak{l}$. Hence by \cite[Theorem 7.5.11(i)]{Neukirch}, we know that the decomposition group at this place has free pro-$\ell$ completion, and hence $H^2(\mathcal{G}_{K_{\mathfrak{l}}}(\ell), \mathbb{Z}/\ell^{f(v, w)}\mathbb{Z})$ vanishes. The infinite place can be disregarded, because this completion equals $\mathbb{C}$.

\emph{Step 2:} Take $v \in \Omega_{\text{fin}}$. Firstly, let us observe that $\tilde{\theta}_v$ is trivial at all places different from $v$, since it ramifies only at $v$. We claim that the restriction of $\tilde{\theta}_v$ to $\mathcal{G}_{K_v}$ is actually a \emph{generator} of
$$
H^2(\mathcal{G}_{K_v}, \mathbb{Z}/\ell^{f(v)}\mathbb{Z}).
$$
Indeed, in $K_v$ we have a $\ell^{f(v)}$-th root of unity. So this last $H^2$ is precisely the $\ell^{f(v)}$-torsion of the Brauer group of $K_v$, which is cyclic of order $\ell^{f(v)}$. Hence it suffices to show that the restriction of $\tilde{\theta}_v$ to $\mathcal{G}_{K_v}$ has order $\ell^{f(v)}$. Notice that we have an isomorphism
$$
\mathcal{G}_{K_v}^{\text{ab}}(\ell) = \mathbb{Z}_\ell \times \mathbb{Z}/\ell^{f(v)}\mathbb{Z}.
$$
The restriction of $\tilde{\theta}_v$ is a generator of $\text{Ext}^1_{\mathbb{Z}_\ell}(\mathbb{Z}_\ell \times \mathbb{Z}/\ell^{f(v)}\mathbb{Z}, \mathbb{Z}/\ell^{f(v)}\mathbb{Z})$, which is cyclic of order $\ell^{f(v)}$, coming from the second coordinate. On the subgroup of such abelian classes, the inflation to $\mathcal{G}_{K_v}$ is injective, thanks to the universal property of the abelianization. This gives that $\tilde{\theta}_v$ is locally at $v$ of order $\ell^{f(v)}$ and hence a generator. 

Observe that we now certainly have that $\tilde{\theta}_v$ generates $H^2(\mathcal{G}_{K_v}, \mathbb{Z}/\ell^{n}\mathbb{Z})$ for $n \leq f(v)$: here we are still using the abuse of notation that $\tilde{\theta}_v$ also denotes any reduction of this class modulo smaller powers of $\ell$.

\emph{Step 3:} Let us distinguish cases. First consider the case that both $v, w \in \Omega_{\text{fin}}$. Then we know that $\pi_v \cup \pi_w$ is trivial at all places different from $v$, $w$ by \emph{Step 1}. Thanks to \emph{Step 2}, we know that we can find $\lambda(v, w), \lambda(w, v) \in \mathbb{Z}/\ell^{f(v, w)}\mathbb{Z}$, such that $\pi_v \cup \pi_w + \lambda(v, w) \cdot \tilde{\theta}_w$ vanishes locally at $w$ and $\pi_v \cup \pi_w +\lambda(w, v) \cdot \tilde{\theta}_v$ vanishes locally at $v$. Furthermore, thanks to \emph{Step 2} once more, we know that $\tilde{\theta}_v$ is locally trivial at $w$ and $\tilde{\theta}_w$ is locally trivial at $v$. Overall we have made sure that 
$$
\pi_v \cup \pi_w + \lambda(w, v) \cdot \tilde{\theta}_v + \lambda(v, w) \cdot \tilde{\theta}_w
$$
is locally trivial at all places. Therefore, invoking \cite[Corollary 9.1.10]{Neukirch}, we conclude that this class actually vanishes in $H^2(\mathcal{G}_K, \mathbb{Z}/\ell^{f(v, w)}\mathbb{Z})$. Thanks to the universal property of the pro-$\ell$ completion it therefore also vanishes in
$$
H^2(\mathcal{G}_K(\ell), \mathbb{Z}/\ell^{f(v, w)}\mathbb{Z}).
$$
Finally, since the class comes from inflation from the abelianization, from the universal property of the maximal $2$-nilpotent quotient we have that the class vanishes in
$$
H^2(\mathcal{G}_K^2(\ell), \mathbb{Z}/\ell^{f(v, w)}\mathbb{Z}),
$$
as desired. It follows from equation (\ref{eExtWedge}) that $H^2(\mathbb{Z}/\ell^{f(v)}\Z \times \mathbb{Z}/\ell^{f(w)}\mathbb{Z}, \mathbb{Z}/\ell^{f(v, w)}\mathbb{Z})$ is generated by $\pi_v \cup \pi_w, \theta_v, \theta_w$, hence we have obtained the desired conclusion. In case precisely one of $v, w$ is not in $\Omega_{\text{fin}}$ (say $v$) the proof follows precisely the same lines. Now the class $\pi_v \cup \pi_w$ is trivial locally everywhere outside of $w$. Locally at $w$ we find $\lambda(v, w)$ such that $\pi_v \cup \pi_w + \lambda(v, w) \cdot \tilde{\theta}_w$ is trivial also at $w$. Now this class is trivial locally everywhere and we finish the argument in the same way as above. In case $f(v) = f(w) = \infty$, we proceed similarly.

It remains to check (C4). Let $w$ be a place in $\Omega_{\text{fin}}$. Let us fix a generator $\sigma_w$ of tame inertia for $\mathcal{G}_{K_w}$ that is sent to $1$ through $\pi_w$. This in particular yields a generator of the image of inertia in $\mathcal{G}_{K_w}(\ell)$, which is, as a profinite group, isomorphic to $\mathbb{Z}_\ell$. Let us fix a Frobenius element $\text{Frob}_w$ at $w$. Recall that we have the crucial relationship
$$
[\text{Frob}_w, \sigma_w] = \sigma_w^{q(w) - 1}. 
$$
In the previous step we established that the central extension given by the $2$-cocycle
$$
\theta := \pi_{v_\ell(1)} \cup \pi_w + \lambda(v_\ell(1), w) \cdot \tilde{\theta}_w,
$$
is realizable by a group epimorphism
$$
\phi(\theta): \mathcal{G}_K \to (\mathbb{Z}/\ell^{f(w)}\mathbb{Z} \times (\mathbb{Z}_\ell \times \mathbb{Z}/\ell^{f(w)}\mathbb{Z}), *_{\theta}),
$$
where the group law is defined in the same way as equation (\ref{eCocycleGroup}). We will now compute $[\text{Frob}_w, \sigma_w] = \sigma_w^{q(w) - 1}$ in two ways. We observe that in the exact sequence (\ref{eExtWedge}), with $A = \mathbb{Z}_\ell \times \mathbb{Z}/\ell^{f(w)}\mathbb{Z}$ and $B = \mathbb{Z}/\ell^{f(w)}\mathbb{Z}$, the cup product maps to the determinant. This implies that $\theta$ must map to the determinant as well, since $\tilde{\theta}_w$ comes from $\text{Ext}(A, B)$ and is therefore trivial in $\text{Hom}(\wedge^2 A, B)$ by exactness. But by definition this map is also the commutator pairing. Therefore to calculate the commutator $[\text{Frob}_w, \sigma_w]$, we need to calculate the determinant of
\[
\begin{pmatrix}
\pi_{v_\ell(1)}(\text{Frob}_w) & \pi_{v_\ell(1)}(\sigma_w) \\
\pi_w(\text{Frob}_w) & \pi_w(\sigma_w)
\end{pmatrix}
=
\begin{pmatrix}
\pi_{v_\ell(1)}(\text{Frob}_w) & 0 \\
\pi_w(\text{Frob}_w) & 1
\end{pmatrix},
\]
which equals the reduction of $\pi_{v_\ell(1)}(\text{Frob}_w)$ modulo $\ell^{f(w)}$.

We will now compute $\sigma_w^{q(w) - 1}$. Consider the map from $H^2(\mathbb{Z}_\ell \times \mathbb{Z}/\ell^{f(w)}\mathbb{Z}, \mathbb{Z}/\ell^{f(w)}\mathbb{Z})$ to $\text{Hom}(\Z/\ell^{f(w)}\Z, \Z/\ell^{f(w)}\Z)$ given by taking $a \in \Z/\ell^{f(w)}\Z$, then lifting $(0, a)$ to the central extension and then powering the lift by $q(w) - 1$. By construction this lands in the center $\Z/\ell^{f(w)}\Z$, and therefore is well-defined. This homomorphism sends the cup product to the zero map, while $\tilde{\theta}_w$ maps $1$ to $1$ by our choice of normalization of $\tilde{\theta}_w$.

Therefore we conclude that $\lambda(v_\ell(1), w)$ equals the reduction of $\pi_{v_\ell(1)}(\text{Frob}_w)$ modulo $\ell^{f(w)}$. This is in turn equal to $\log_\ell(\pi_{\text{cyc}}(\text{Frob}_w)) = \log_\ell(q(w)) \equiv \log_\ell(1 + \ell^{f(w)} \cdot g(w)) \bmod \ell^{f(w)}$, as desired.
\end{proof}

We now make the key calculation of the labels of the graph $\Phi_{\mathcal{S}_K}$. This is a minor generalization of the calculation at the end of the last proof.

\begin{proposition} 
\label{Computing labels}
Let $v, w \in \Omega$ with $v < w$. Suppose that $w \in \Omega_{\textup{fin}}$. Then
$$
\lambda_{\mathcal{S}_K}(v, w) = \pi_v(\textup{Frob}_w). 
$$
Further suppose that $v \in \Omega_{\textup{fin}}$. Then
$$ 
\lambda_{\mathcal{S}_K}(w, v) = -\pi_w(\textup{Frob}_v). 
$$
\end{proposition}

\begin{proof}
Let us fix $\sigma_w$, $\text{Frob}_w$ as at the end of the proof of Proposition \ref{Galois groups are ok}. We now plug the relationship
$$
[\text{Frob}_w, \sigma_w] = \sigma_w^{q(w) - 1}
$$
into the continuous homomorphism from $\mathcal{G}_K$ to the group
$$
(\mathbb{Z}/\ell^{f(v, w)} \times (\mathbb{Z}/\ell^{f(v)}\mathbb{Z}\times \mathbb{Z}/\ell^{f(w)}\mathbb{Z}),*_{\theta}),
$$
where 
$$
\theta := \pi_v \cup \pi_w + \lambda_{\mathcal{S}_K}(w, v) \cdot \tilde{\theta}_v + \lambda_{\mathcal{S}_K}(v, w) \cdot \tilde{\theta}_w.
$$
The left hand side yields $\pi_v(\text{Frob}_w)$, while the right hand side yields $\lambda_{\mathcal{S}_K}(v, w)$, owing to our normalization of the extension classes. Hence we obtain
$$
\lambda_{\mathcal{S}_K}(v, w) = \pi_v(\text{Frob}_w).
$$
Let now $v$ be in $\Omega_{\text{fin}}$. The argument goes precisely the same way, with the only exception that when we plug in $[\text{Frob}_v, \sigma_v]$ we get $-\pi_w(\text{Frob}_v)$. Indeed, we now consider the determinant of the matrix whose first column equals $(\pi_v(\text{Frob}_v), \pi_w(\text{Frob}_v))$, while the second column equals $(\pi_v(\sigma_v), \pi_w(\sigma_v)) = (1, 0)$.
\end{proof}

We now introduce certain auxiliary Galois groups that will aid us in obtaining the desired level of control on the labels of the graph. Let us first fix some notation. Let $v$ be a finite place of $k$ coprime to $\ell$ and write $\mathfrak{p}$ for the corresponding prime ideal. We fix any generator $\alpha_v$ of $\mathfrak{p}^{\#\text{Cl}(K)}$ (there are only two choices, since we know by assumption on $K$ that $1, -1$ are the only roots of unity inside of $K$). Observe that $\text{Cl}(K)[\ell] = \{0\}$ implies that $\alpha_v$ has $v$-adic valuation coprime with $\ell$: this will play a key role in the proof of Corollary \ref{Galois groups are Rado}.

Let $m$ now be a positive integer. For each $v \in \Omega_{\text{fin}}$, we choose a fixed element $\alpha_v^{\frac{1}{\ell^m}}$ of $K^{\text{sep}}$, whose $\ell^m$-th power is $\alpha_v$. Let $S_0$ be a finite subset of $\Omega_{\text{fin}}$. We write $\text{red}_{\ell^m}: \Z_\ell \rightarrow \Z/\ell^m\Z$ for the natural reduction map. We define
$$
K(S_0, m)
$$
to be the compositum of the fields $K(\text{red}_{\ell^{2m}} \circ \pi_{v_\ell(1)})$, $K(\text{red}_{\ell^m} \circ \pi_{v_\ell(2)})$, $K(\{\alpha_v^{\frac{1}{\ell^{\min(m, f(v))}}} : v \in S_0\})$ and $K(\{\pi_{v} : v \in S_0\})$. Here the first, second and fourth fields are the fixed fields of the respective characters. We also fix a choice of an $\ell^m$-th root of unity $\zeta_{\ell^m}$ and denote by $i_m$ the unique identification of $\langle \zeta_{\ell^m} \rangle$ and $\mathbb{Z}/\ell^{f(m)}\mathbb{Z}$ sending $\zeta_{\ell^m}$ to $1$. 

This notation having been fixed, we will define a continuous homomorphism
$$
\pi(S_0, m): \mathcal{G}_K \to \left(\prod_{v \in S_0}\mathbb{Z}/\ell^{\min(m, f(v))}\mathbb{Z}\right)^2 \rtimes ((\mathbb{Z}/\ell^{2m}\mathbb{Z})^\ast \times \mathbb{Z}/\ell^m\mathbb{Z}).
$$
Let us explain the yet unspecified action on the target. We let $\mathbb{Z}/\ell^m\mathbb{Z}$ act trivially, while $(\mathbb{Z}/\ell^{2m}\mathbb{Z})^\ast$ acts by a joint diagonal scaling on the first copy of $\prod_{v \in S_0}\mathbb{Z}/\ell^{\min(m, f(v))}\mathbb{Z}$ and trivially on the second copy. We define a map by sending $\sigma \in \mathcal{G}_K$ to the vector
$$
\left(\left(i_m\left(\frac{\sigma(\alpha_v^{\frac{1}{\ell^{\min(m, f(v))}}})}{\alpha_v^{\frac{1}{\ell^{\min(m, f(v))}}}}\right)\right)_{v \in S_0},\left(\pi_v(\sigma)\right)_{v \in S_0}, \text{red}_{\ell^{2m}}(\pi_{\text{fcyc}}(\sigma)), \text{red}_{\ell^m}(\pi_{v_\ell(2)}(\sigma))\right),
$$
where $\pi_{\text{fcyc}}$ now denotes the full cyclotomic character $\mathcal{G}_K \rightarrow \hat{\Z}^\ast$ (unlike earlier, where $\pi_{\text{cyc}}$ denoted a related but slightly different map). This map induces a group isomorphism
$$
\pi(S_0, m): \text{Gal}(K(S_0,m)/K) \to \left(\prod_{v \in S_0}\mathbb{Z}/\ell^{\min(m, f(v))}\mathbb{Z}\right)^2 \rtimes ((\mathbb{Z}/\ell^{2m}\mathbb{Z})^\ast\times \mathbb{Z}/\ell^m\mathbb{Z}).
$$
Observe that for the surjectivity of the map it is crucial that the fields 
$$
K(\{\alpha_v^{\frac{1}{\ell^{\min(m, f(v))}}} : v \in S_0\}), \quad K(\{\pi_v : v \in S_0\})
$$
are linearly disjoint, which follows from the fact that $K$ does not have $\ell$-th roots of unity.

\begin{corollary} 
\label{Galois groups are Rado}
We have that $\mathcal{S}_K$ is weakly Rado.  
\end{corollary}

\begin{proof}
Let us fix a finite subset $S_0$ of $\Omega_{\text{fin}}$ and let us put $S := \{v_\ell(1), v_\ell(2)\} \cup S_0$. Let us fix the choices of $n$, $\alpha$, $((\lambda_s(1))_{s \in S}, (\lambda_s(2))_{s \in S})$ as in Definition \ref{Weakly Rado graphs}. We have to show that there exist $v \in \Omega - S$ and $\gamma \in (\mathbb{Z}/\ell^n\mathbb{Z})^\ast$ satisfying $f(v) = n$, $g(v) = \alpha$ and
$$
\lambda(v, s) = \gamma \cdot \lambda_s(1), \quad \lambda(s, v) = \lambda_s(2)
$$
for each $s \in S$.

To this end, let us consider the element
\begin{multline*}
x := ((\lambda_s(1))_{s \in S_0}, (\lambda_s(2))_{s \in S_0}, 1 + \ell^n \cdot \alpha, \lambda_{v_\ell(2)}(2)) \in \\
\left(\prod_{v \in S_0} \mathbb{Z}/\ell^{\min(n, f(v))}\mathbb{Z}\right)^2 \rtimes ((\mathbb{Z}/\ell^{2n}\mathbb{Z})^\ast\times \mathbb{Z}/\ell^n\mathbb{Z}).
\end{multline*}
Observe that the conjugacy class of $x$, which we will denote by $\mathcal{C}(x)$, consists precisely of the elements of the form
$$
x := (\gamma_0 \cdot (\lambda_s(1))_{s \in S_0}, (\lambda_s(2))_{s \in S_0}, 1 + \ell^n \cdot \alpha, \lambda_{v_\ell(2)}(2))
$$
as $\gamma_0$ runs through $(\mathbb{Z}/\ell^n\mathbb{Z})^\ast$.

Thanks to the Chebotarev density theorem, there exists a place $v \in \Omega - S$ such that
$$
\pi(S_0, n)(\text{Frob}_v) \in \mathcal{C}(x)
$$
and $v > s$ for each $s \in S$. Now combining Proposition \ref{Computing labels} with the fact that $v > s$ for each $s \in S$ and that $\pi(S_0, n)(\text{Frob}_v) \in \mathcal{C}(x)$, we see that
$$
f(v) = n, \quad g(v) = \alpha, \quad \lambda(s, v) = \lambda_s(2)
$$
for each $s \in S$. Let us denote by $\gamma_1 := \#\text{Cl}(K) \in \mathbb{Z}_\ell^\ast$. The fact that $\gamma_1$ is a $\ell$-adic unit is precisely due to the fact that $\text{Cl}(K)[\ell] = \{0\}$. Then we have that
$$
\gamma_1 \cdot (\pi_v(\text{Frob}_s))_{s \in S_0} = (\pi_v(\text{Art}(\alpha_s)))_{s \in S_0}.
$$
Now a moment reflection shows that there exists $\gamma_2 \in \mathbb{Z}_\ell^\ast$ such that
$$
\gamma_2 \cdot (\pi_v(\text{Art}(\alpha_s)))_{s \in S_0} = \rho(\pi(S_0, n)(\text{Frob}_v)),
$$
where $\rho$ is the projection on (the first copy of) $\prod_{v \in S_0} \mathbb{Z}/\ell^{\min(n, f(v))}\mathbb{Z}$. Hence invoking the description of the conjugacy class $\mathcal{C}(x)$, we see that there exists $\gamma_3 \in \mathbb{Z}_\ell^\ast$ such that
$$
\gamma_3 \cdot \rho(\pi(S_0, n)(\text{Frob}_v)) = (\lambda_s(1))_{s \in S_0}.
$$
Now, combining with Proposition \ref{Computing labels}, we conclude that
$$
-\gamma_1 \cdot \gamma_2 \cdot \gamma_3 \cdot (\lambda(v, s))_{s \in S_0} = (\lambda_s(1))_{s \in S_0}.
$$
Therefore, by construction of the labels $\lambda_s(1)$ for $s \in \{v_\ell(1), v_\ell(2)\}$, we have that
\[
-\gamma_1 \cdot \gamma_2 \cdot \gamma_3 \cdot (\lambda(v, s))_{s \in S} = (\lambda_s(1))_{s \in S}.
\]
This ends the required verification (since we checked $\gamma_1, \gamma_2, \gamma_3 \in \mathbb{Z}_\ell^\ast$) and hence the proof. 
\end{proof}

We are now ready to prove the main result of this section. 

\begin{theorem}
\label{tMain1}
Let $K_1$, $K_2$ be two imaginary quadratic number fields with class number not divisible by $\ell$. For $\ell = 3$, further suppose that $K_i \otimes_{\mathbb{Q}} \mathbb{Q}_3$ does not have a primitive third root of unity for $i \in \{1, 2\}$. Then 
$$
\mathcal{G}_{K_1}^2(\ell) \simeq_{\textup{top.gr.}} \mathcal{G}_{K_2}^2(\ell). 
$$
More precisely, both are isomorphic, as profinite groups, to $\mathcal{G}(\Phi(\textup{Rado}))$. 
\end{theorem}

\begin{proof}
This is a consequence of Corollary \ref{Galois groups are Rado} and Theorem \ref{weakly Rado groups are isomorphic}.   
\end{proof}

\section{\texorpdfstring{Galois groups for $\ell = 2$}{Galois groups for 2}} \label{section: Galois groups at 2}
We now adapt the material of Section \ref{section: Galois groups l odd} to the case $\ell := 2$. We let $K$ be an imaginary quadratic field such that $2$ is inert and $\text{Cl}(K)[2] = \{0\}$. Thanks to Gauss' genus theory, this is equivalent to saying that $K:=\Q(\sqrt{-p})$, where $p$ is an odd positive prime with $p \equiv 3 \bmod 8$. 

In the rest of this section $v_2(1)$ denotes the cyclotomic $\mathbb{Z}_2$-extension of $K$ and $v_2(2)$ denotes a fixed $\mathbb{Z}_2$-extension of $K$ having $K(\sqrt{-1})$ as the quadratic subextension. We define $\Omega := \{v_2(1),v_2(2)\} \cup \Omega_{\text{fin}}$, where $\Omega_{\text{fin}}$ consists of the finite places $v$ of $K$ that are coprime to $2$. We define $f(v)$ to be the $2$-adic valuation of $\#(\mathcal{O}_K/v)^{\ast}$ for $v \in \Omega_{\text{fin}}$ and we define $f(v) := \infty$ for $v \in \{v_2(1), v_2(2)\}$. 

Before moving on, we shall need two basic facts about the local cohomology of $K_2 \simeq_{\Q_2-\text{alg.}} \Q_4$, the completion of $K$ at the prime $2$. The following fact will be handy in normalizing our choices of characters locally at $2$ as well as in the computation of Hilbert symbols taking place in the proof of Corollary \ref{Galois groups are reciprocity groups}.

\begin{proposition} 
\label{Hilbert symbols at K2}
Let $\epsilon := 2 + \sqrt{5} \in \mathcal{O}_{K_2}^\ast$. Then $\{\epsilon, -1, 2\sqrt{5} - 5\}$ is a basis of $\frac{\mathcal{O}_{K_2}^\ast}{\mathcal{O}_{K_2}^{2\ast}}$. Furthermore, the Hilbert symbol pairing is given in this basis by the matrix 
\[
\begin{pmatrix}
1 & 1 & 0 \\
1 & 0 & 0 \\
0 & 0 & 0
\end{pmatrix}.
\]
\end{proposition}

\begin{proof}
This is a direct verification.
\end{proof}

The next proposition gives us a good explicit handle on classes in the local cohomology groups $H^2(\mathcal{G}_{K_2}, \Z/2^n\Z)$.

\begin{proposition} 
\label{local cohomology at 2}
Let $n$ be a positive integer. Then the reduction map $H^2(\mathcal{G}_{K_2}, \Z/2^n\Z) \to H^2(\mathcal{G}_{K_2}, \Z/2\Z)$ is an isomorphism. In particular, $H^2(\mathcal{G}_{K_2}, \Z/2^n\Z)$ is cyclic of order $2$. 
\end{proposition}

\begin{proof}
By local Tate duality, the last part is equivalent to $H^0(\mathcal{G}_{K_2}, \mu_{2^n})$ being cyclic of order $2$, which is indeed the case. The first part is trivial for $n = 1$. From now on we assume that $n > 1$ and proceed by induction on $n$. Local class field theory shows that 
$$
\frac{H^1(\mathcal{G}_{K_2}, \Z/2\Z)}{2^{n - 1} \cdot H^1(\mathcal{G}_{K_2}, \Z/2^n\Z)}
$$ 
is cyclic of order $2$. Applying cohomology to the short exact sequence
$$
0 \rightarrow \Z/2^{n - 1}\Z \rightarrow \Z/2^n\Z \rightarrow \Z/2\Z \rightarrow 0
$$
and using that $H^2(\mathcal{G}_{K_2}, \Z/2^{n - 1}\Z)$ has order $2$ (by induction), we obtain that the reduction map is injective and therefore an isomorphism.
\end{proof}

\begin{lemma}
\label{lLocalGlobal}
Let $n$ be a positive integer and let $\theta \in H^2(\mathcal{G}_K, \Z/2^n\Z)$. Suppose that $\theta$ is locally trivial at all odd places. Then $\theta$ is trivial in $H^2(\mathcal{G}_K, \Z/2^n\Z)$.
\end{lemma}

\begin{proof}
Denote by $c$ the reduction of $\theta$ modulo $2$. Then $c$ is trivial at all odd places, and therefore also trivial at $2$ by Hilbert reciprocity. Since the reduction map $H^2(\mathcal{G}_{K_2}, \Z/2^n\Z) \to H^2(\mathcal{G}_{K_2}, \Z/2\Z)$ is an isomorphism by Proposition \ref{local cohomology at 2} and maps the restriction of $\theta$ to the restriction of $c$, we conclude that the restriction of $\theta$ is trivial in $H^2(\mathcal{G}_{K_2}, \Z/2^n\Z)$. We now invoke \cite[Corollary 9.1.10]{Neukirch}. Observe that we are not in the ``special case'' (i.e. Grunwald--Wang) of that corollary, because $2$ is totally ramified in $K(\zeta_{2^k})/K$, as $2$ is inert in $K$. Then \cite[Corollary 9.1.10]{Neukirch} shows that $\theta$ is globally trivial, as desired.
\end{proof}

We can now give the analogue of Proposition \ref{a totally ramified extension} for $\ell = 2$.

\begin{proposition}
\label{a totally ramified extension at 2}
For each $v \in \Omega_{\textup{fin}}$, there exists a character $\pi_v: \mathcal{G}_K \twoheadrightarrow \mathbb{Z}/2^{f(v)}\mathbb{Z}$ that is totally ramified at $v$, unramified at all the other odd places and that at $2$ decomposes as 
$$
\pi_v|_{\mathcal{G}_{K_2}} = \pi_v(\textup{unr}) + \chi_\epsilon,
$$
where $\pi_v(\textup{unr})$ is unramified and $\chi_\epsilon$ is the ramified character corresponding to $\epsilon := 2 + \sqrt{5}$ from Proposition \ref{Hilbert symbols at K2}. Furthermore, $2\cdot \pi_v$ gives the largest abelian $2$-power extension of $K$ unramified outside of $v$. 
\end{proposition}

\begin{proof}
A straighforward adaptation of the proof of Proposition \ref{a totally ramified extension} gives that the largest abelian $2$-power extension of $K$ unramified outside of $v$ is given by a surjective character
$$
\pi_v(0): \mathcal{G}_K \to \Z/2^{f(v) - 1}\Z.
$$
The problem of lifting this character to $\Z/2^{f(v)}\Z$ defines a central $\mathbb{F}_2$-extension and hence a class in $H^2(\mathcal{G}_K,\mathbb{F}_2)$. This class is unramified at all places different from $v$ and hence it is trivial also at $v$ by Hilbert reciprocity. Therefore thanks to the local-global principle for Brauer classes (given that $\mathbb{F}_2 \simeq \mu_2$) we obtain that the class is trivial globally. Let $\pi_v(1)$ be the character so obtained. 

We let $\gamma_{\mathfrak{q}}$ be a generator of $\mathfrak{q}^{\#\text{Cl}(K)}$ for each odd prime $\mathfrak{q}$. Because $\text{Cl}(K)[2] = \{0\}$, the quadratic character $\chi_{\gamma_{\mathfrak{q}}}$ ramifies at $\mathfrak{q}$ and potentially also at $2$. Using these characters, we twist the original character $\pi_v(1)$ to remove all the ramification outside of $v$ and $2$. The resulting character differs locally at $2$ from an unramified character via a quadratic character. Twisting with $\chi_2$ we ensure the resulting character comes from taking the square root of an element of $\mathcal{O}_{K_2}^\ast$ locally at $2$. After potentially twisting by $\chi_{-1}$, Proposition \ref{Hilbert symbols at K2} guarantees that the resulting character $\pi_v$ has the shape
$$
\pi_v|_{\mathcal{G}_{K_2}} = \pi_v(\textup{unr}) + \lambda \cdot \chi_\epsilon
$$
for some $\lambda \in \mathbb{F}_2$. We claim that $\lambda = 1$. But indeed, if not, then $\pi_v$ gives an abelian $2$-power extension of $K$ unramified outside of $v$ of order $2^{f(v)}$, contradicting the definition of $\pi_v(0)$.
\end{proof}

For each $v \in \Omega_{\text{fin}}$, fix a $\pi_v$ as guaranteed by Proposition \ref{a totally ramified extension at 2}. Also fix
$$
\pi_{v_2(1)} := \log_2 \circ \pi_{\text{cyc}}: \mathcal{G}_K \to \mathbb{Z}_2,
$$
where $\pi_{\text{cyc}}$ is the cyclotomic character from $\mathcal{G}_K$ to $1 + 4 \cdot \mathbb{Z}_2 = \frac{\Z_2^{\ast}}{\langle -1 \rangle}$ with this last identification coming from the natural projection map. Since $K = \Q(\sqrt{-p})$ with $p$ odd, $K$ is linearly disjoint from the unique $\mathbb{Z}_2$-extension of $\mathbb{Q}$. Therefore we deduce that $\pi_{v_2(1)}$ is \emph{surjective}. 

Recall that we fixed any surjective character $\pi_{v_2(2)}: \mathcal{G}_K \twoheadrightarrow \mathbb{Z}_2$ that reduced modulo $2$ coincides with $\chi_{-1}$: it is actually not hard to use the fact that $\text{Cl}(K)[2] = \{0\}$ to deduce that the largest pro-$2$ extension of $K$ unramified outside of $2$ has Galois group isomorphic to the free pro-$2$ group on two generators, which in particular implies the existence of this extension. 

Since $K$ is quadratic and unramified at $2$, we have that the two characters $\chi_2$, $\chi_{-1}$ are still linearly independent over $K$. Hence, by Nakayama's lemma, we have that the pair $(\pi_{v_2(1)},\pi_{v_2(2)})$ defines a surjective homomorphism onto $\mathbb{Z}_2^2$. The corresponding field is both the unique $\Z_2^2$-extension of $K$ as well as the largest pro-$2$ abelian extension of $K$ unramified outside of $2$. We can now package this construction into a single continuous homomorphism
$$
\pi_\bullet: \mathcal{G}_K \to \prod_{v \in \Omega} \mathbb{Z}/2^{f(v)}\mathbb{Z}.
$$

\begin{proposition} 
\label{describing abelianization at 2}
The homomorphism $\pi_\bullet$ is surjective. Furthermore, $\pi_\bullet$ factors through $\mathcal{G}_K(2)$ and induces an isomorphism between $\mathcal{G}_K(2)/[\mathcal{G}_K(2), \mathcal{G}_K(2)]$ and $\prod_{v \in \Omega} \mathbb{Z}/2^{f(v)}\mathbb{Z}$.
\end{proposition}

\begin{proof}
This is a straightforward adaptation of the proof of Proposition \ref{describing abelianization}.
\end{proof}

For each $w \in \Omega_{\text{fin}}$, let us put
$$
q(w) := \# \mathcal{O}_K/w.
$$
We write $g(w)$ for the unique element of $(\mathbb{Z}/2^{f(w)}\mathbb{Z})^\ast$ satisfying
$$
q(w) \equiv 1+ 2^{f(w)} \cdot g(w) \bmod 2^{2f(w)}.
$$

\begin{proposition}
\label{Galois groups are ok at 2}
The pair $\mathcal{S}_K := (\mathcal{G}_K^2(2), \pi_\bullet)$ is a $2$-nilpotent pro-$2$ group on $(\Omega, f, g)$.  
\end{proposition}

\begin{proof}
The proof of properties (C1) and (C2) goes precisely as in the proof of Proposition \ref{Galois groups are ok}. We will now explain the relevant changes in the three steps of the verification of (C3).

\emph{Step 1:} Arguing as in Proposition \ref{Galois groups are ok}, we observe that $\pi_v \cup \pi_w$ is locally trivial at all places different from $v, w$ and $2$. 

\emph{Step 2:} Take $v \in \Omega_{\text{fin}}$. Firstly, let us observe that $\tilde{\theta}_v$ is trivial at all places different from $v$ and $2$, since it could be ramified only at $v$ and at $2$. Following the proof of Proposition \ref{Galois groups are ok}, we see that the restriction of $\tilde{\theta}_v$ to $\mathcal{G}_{K_v}$ is actually a \emph{generator} of $H^2(\mathcal{G}_{K_v}, \mathbb{Z}/2^{f(v)}\mathbb{Z})$.

\emph{Step 3:} We will distinguish cases. First consider the case that both $v, w \in \Omega_{\text{fin}}$. Following the proof of Proposition \ref{Galois groups are ok}, we find $\lambda(w, v)$ and $\lambda(v, w)$ such that 
$$
c := \pi_v \cup \pi_w + \lambda(w, v) \cdot \tilde{\theta}_v + \lambda(v, w) \cdot \tilde{\theta}_w
$$
is locally trivial at all places different from $2$. Therefore $c$ is globally trivial by Lemma \ref{lLocalGlobal}. Arguing as in Proposition \ref{Galois groups are ok}, we conclude that $c$ also vanishes in
$$
H^2(\mathcal{G}_K^2(2), \mathbb{Z}/2^{f(v, w)}\mathbb{Z}).
$$
As explained in Subsection \ref{intermezzo}, the cohomology group $H^2(\mathbb{Z}/2^{f(v)}\Z \times \mathbb{Z}/2^{f(w)}\mathbb{Z}, \mathbb{Z}/2^{f(v, w)}\mathbb{Z})$ is generated by $\pi_v \cup \pi_w$, $\theta_v$ and $\theta_w$, hence we have obtained the desired conclusion. In case $v$ or $w$ is not in $\Omega_{\text{fin}}$, the proof follows precisely the same lines.

It remains to sketch the necessary modifications for (C4). Following Proposition \ref{Galois groups are ok}, we conclude that $\lambda(v_2(1), w)$ equals the reduction of $\pi_{v_2(1)}(\text{Frob}_w)$ modulo $2^{f(w)}$. This is in turn equal to $\log_2(q(w)) \equiv \log_2(1+ 2^{f(w)} \cdot g(w)) \bmod 2^{f(w)}$ in case $f(w) \geq 2$ as desired. Finally, by the same argument, the label $\lambda(v_2(2), w)$ is congruent to $\pi_{v_2(2)}(\text{Frob}_w)$ modulo $2^{f(w)}$, which reduced modulo $2$ is precisely $\chi_{-1}(\text{Frob}_w)$. This vanishes if and only if $f(w) \geq 2$, completing the proof.
\end{proof}

We are now ready to compute the labels of the graph $\Phi_{\mathcal{S}_K}$.

\begin{proposition} 
\label{Computing labels at 2}
Let $v, w \in \Omega$ with $v < w$. Suppose that $w \in \Omega_{\textup{fin}}$. Then
$$
\lambda_{\mathcal{S}_K}(v, w) = \pi_v(\textup{Frob}_w). 
$$
Further suppose that $v \in \Omega_{\textup{fin}}$. Then
$$ 
\lambda_{\mathcal{S}_K}(w, v) = -\pi_w(\textup{Frob}_v). 
$$
\end{proposition}

\begin{proof}
Identical to the proof of Proposition \ref{Computing labels}. 
\end{proof}

We finally show that our group is a reciprocity group. 

\begin{corollary} 
\label{Galois groups are reciprocity groups}
The group $\mathcal{S}_K$ is a reciprocity $2$-nilpotent pro-$2$ group on $(\Omega, f, g)$.
\end{corollary}

\begin{proof}
Thanks to Proposition \ref{Computing labels at 2}, this comes down to the equation
$$
\pi_v(\text{Frob}_w) \equiv \pi_w(\text{Frob}_v) + \frac{q_v - 1}{2}\cdot \frac{q_w - 1}{2} \bmod 2
$$
for $v, w \in \Omega_{\text{fin}}$. Let us denote by $\chi_v$ and by $\chi_w$ the reduction modulo $2$ of $\pi_v$ and $\pi_w$ respectively. By definition we have that $\pi_v(\text{Frob}_w) \equiv \chi_v(\text{Frob}_w) \bmod 2$ and $\pi_w(\text{Frob}_v) \equiv \chi_w(\text{Frob}_v) \bmod 2$. These quadratic characters are unramified outside of $v,w$ and $2$. Hence Hilbert reciprocity yields
$$
(\chi_v, \chi_w)_v + (\chi_v, \chi_w)_w + (\chi_v, \chi_w)_2 = 0.
$$
The first two terms are precisely $\chi_w(\text{Frob}_v)$ and $\chi_v(\text{Frob}_w)$ respectively. Assuming that $\max(f(v), f(w)) \geq 2$, then one of these two characters is unramified locally at $2$ and the other is given by the square root of a unit locally at $2$ thanks to Proposition \ref{a totally ramified extension at 2}. Therefore the Hilbert symbol vanishes thanks to Proposition \ref{Hilbert symbols at K2}. Likewise, precisely owing to the fact that $\max(f(v), f(w)) \geq 2$, we see that $\frac{q_v - 1}{2} \cdot \frac{q_w - 1}{2}$ vanishes modulo $2$, as desired. This proves the desired conclusion when $\max(f(v), f(w)) \geq 2$. 

Let us now examine the case that $f(v) = f(w) = 1$. Proposition \ref{a totally ramified extension at 2} yields that 
$$
\chi_v = \pi_v|_{\mathcal{G}_{K_2}} = \chi_\epsilon + \pi_v(\text{unr}), \quad \chi_w = \pi_w|_{\mathcal{G}_{K_2}} = \chi_\epsilon + \pi_w(\text{unr}),
$$
where $\pi_v(\text{unr}), \pi_w(\text{unr})$ are now \emph{quadratic} unramified characters, while $\chi_{\epsilon}$ is the character described in Proposition \ref{Hilbert symbols at K2}. Proposition \ref{Hilbert symbols at K2} and bilinearity imply that
$$
(\chi_v, \chi_w)_2 = (\chi_\epsilon, \chi_\epsilon)_2 + (\chi_\epsilon, \pi_w(\text{unr}))_2 + (\pi_v(\text{unr}), \chi_\epsilon) + (\pi_v(\text{unr}), \pi_w(\text{unr})) = 1 + 0 + 0 + 0 = 1,
$$
precisely as desired.
\end{proof}

Following the construction of the field $K(S_0, m)$ in Section \ref{section: Galois groups l odd}, we find a slightly smaller Galois group, due to the presence of a $2$-nd root of unity in $K$. When applying Chebotarev, this gives precisely the extra constraints of a decorated \emph{reciprocity} graph. Hence following the proof of Corollary \ref{Galois groups are Rado}, we get our next result.

\begin{corollary} 
\label{Galois groups are Rado at 2}
We have that $\mathcal{S}_K$ is weakly Rado.  
\end{corollary}

This implies the main result of this section. 

\begin{theorem}
\label{tMain2}
Let $p_1$, $p_2$ be two positive prime numbers congruent to $3$ modulo $8$. Then 
$$
\mathcal{G}_{\Q(\sqrt{-p_1})}^2(2) \simeq_{\textup{top.gr.}} \mathcal{G}_{\Q(\sqrt{-p_2})}^2(2). 
$$
More precisely, both are isomorphic, as profinite groups, to $\mathcal{G}(\Phi(\textup{Rado}))$. 
\end{theorem}

\begin{proof}
This is a consequence of Corollary \ref{Galois groups are Rado at 2} and Theorem \ref{weakly Rado groups are isomorphic at 2}.   
\end{proof}

\section{Proof of main theorems}
\label{sProof}

\subsection{Proof of Theorem \ref{Thm: main2}}
Theorem \ref{Thm: main2} follows from Theorem \ref{tMain1} and Theorem \ref{tMain2}.

\subsection{Proof of Corollary \ref{Cor: main3} and Theorem \ref{Thm: main1}}
Corollary \ref{Cor: main3} follows from Theorem \ref{Thm: main2} upon observing that $\Q(\sqrt{-11})$, $\Q(\sqrt{-19})$, $\Q(\sqrt{-43})$, $\Q(\sqrt{-67})$ and $\Q(\sqrt{-163})$ have class number $1$ and are of the shape $\Q(\sqrt{-p})$ with $p \equiv 3 \bmod 8$ and recalling that, if $K$ is a field, the natural map
$$\mathcal{G}_K^{2} \to \prod_{\ell \ \text{prime}}\mathcal{G}_K^2(\ell)
$$ 
is an isomorphism. Theorem \ref{Thm: main1} follows from Corollary \ref{Cor: main3}.

\subsection{Proof of Theorem \ref{Thm: reconstruction} and Corollary \ref{cQ7}}
Part $(a)$ of Theorem \ref{Thm: reconstruction} is an immediate consequence of Proposition \ref{reconstructing cyclotomic abstract} and Corollary \ref{Galois groups are Rado}. Likewise, part $(b)$ follows at once from Proposition \ref{reconstructing cyclotomic abstract at 2} and Corollary \ref{Galois groups are Rado at 2}. Corollary \ref{cQ7} follows rather directly from the following theorem, as $\chi_{-1}$ does not lift to a $\Z/4\Z$ character for $\Q(\sqrt{-7})$, but it does for the other fields. 

\begin{theorem}
\label{chi-1 in general}
Let $K$ be a number field. Let $\chi(0): \mathcal{G}_K \to \mathbb{F}_2$ be a continuous homomorphism such that for all $\chi \in \textup{Hom}_{\textup{top.gr.}}(\mathcal{G}_K, \mathbb{F}_2)$, we have that $\chi_0 \cup \chi$ vanishes in $H^2(\mathcal{G}_K,\mathbb{F}_2)$ if and only if $\chi \in 2 \cdot \textup{Hom}_{\textup{top.gr.}}(\mathcal{G}_K, \Z/4\Z)$. Then $\chi(0) = \chi_{-1}$. 

In particular,
$$
\textup{ker}(\chi_{-1}: \mathcal{G}_K^2(2) \to \mathbb{F}_2)
$$
is invariant under $\textup{Aut}_{\textup{top.gr.}}(\mathcal{G}_K^2(2))$.    
\end{theorem}

\begin{proof}
Let $\chi \in \textup{Hom}_{\textup{top.gr.}}(\mathcal{G}_K, \mathbb{F}_2)$. Using the identity $\chi_{-1} \cup \chi = \chi \cup \chi$, we get
\begin{align*}
\chi_{-1} \cup \chi \text{ vanishes in } H^2(\mathcal{G}_K, \mathbb{F}_2) 
&\Longleftrightarrow \chi \cup \chi \text{ vanishes in } H^2(\mathcal{G}_K, \mathbb{F}_2) \\
&\Longleftrightarrow \chi \in 2 \cdot \textup{Hom}_{\textup{top.gr.}}(\mathcal{G}_K, \Z/4\Z).
\end{align*}
Let us first suppose that $\chi_{-1} = 0$. Then we have that  
$$
\textup{Hom}_{\textup{top.gr.}}(\mathcal{G}_K, \mathbb{F}_2) = 2 \cdot \textup{Hom}_{\textup{top.gr.}}(\mathcal{G}_K, \Z/4\Z).
$$
Furthermore, $\chi(0) \cup \chi$ vanishes in $H^2(\mathcal{G}_K,\mathbb{F}_2)$ for each $\chi \in \textup{Hom}_{\textup{top.gr.}}(\mathcal{G}_K, \mathbb{F}_2)$. Since the local Hilbert pairing is non-degenerate, it follows that $\chi(0)$ restricts to the trivial character at every place of $K$. This forces $\chi(0) = 0$ by local-to-global \cite[Corollary 9.1.10]{Neukirch}. Performing the same argument with roles reversed, we obtain that $\chi(0) = 0$ implies $\chi_{-1} = 0$. 

From now on we assume that $\chi_{-1}$ and $\chi(0)$ are both non-zero. Hence, by contradiction, we may assume that $\chi_{-1}$ and $\chi(0)$ are linearly independent. We obtain an immediate contradiction with the following claim. 

\emph{Claim:} If $\chi(1)$, $\chi(2)$ are linearly independent characters in $\textup{Hom}_{\textup{top.gr.}}(\mathcal{G}_K, \mathbb{F}_2)$, then there exists a character $\chi(3) \in \textup{Hom}_{\textup{top.gr.}}(\mathcal{G}_K, \mathbb{F}_2)$ such that $\chi(1) \cup \chi(3)$ vanishes and $\chi(2) \cup \chi(3)$ does not vanish. 

\emph{Proof of Claim:} Let $S$ be the finite set of places where $\chi(1)$ or $\chi(2)$ ramifies together with the places of $K$ above $2$ and $\infty$. By the Chebotarev density theorem, we have an infinite set $T$ of prime ideals that are not in $S$, such that $\chi(1)(\text{Frob}_{\mathfrak{p}}) = 0$ and $\chi(2)(\text{Frob}_{\mathfrak{p}}) = 1$. Repeatedly applying the pigeonhole principle, we find an infinite sequence $(S_n)_{n \geq 1}$ of disjoint finite non-empty sets inside $T$, with the property that
$$
\prod_{\mathfrak{p} \in S_n} \mathfrak{p} = 0 
$$
in $\text{Cl}(K)$ for each $n \geq 1$. Let us fix, for each $n \geq 1$, a generator $\gamma_n$ of $\prod_{\mathfrak{p} \in S_n} \mathfrak{p}$. Yet another application of the pigeonhole principle yields two distinct positive integers $n_1$, $n_2$ such that locally at the primes in $S$ the elements $\gamma_{n_1}, \gamma_{n_2}$ land in the same class modulo squares. We claim that $\chi(3) := \chi_{\gamma_{n_1}\gamma_{n_2}}$ satisfies the desired conclusion. Since $\chi(2) \cup \chi(3)$ is not locally trivial at the places in $S_{n_1}$ and $S_{n_1}$ is non-empty, $\chi(2) \cup \chi(3)$ is certainly not globally trivial.

It remains to show that $\chi(1) \cup \chi(3)$ is trivial, to which end it suffices to show that $\chi(1) \cup \chi(3)$ is locally trivial everywhere by \cite[Corollary 9.1.10]{Neukirch}. By construction, $\chi(3)$ is the trivial character locally at all places in $S$ and hence $\chi(1) \cup \chi(3)$ vanishes locally at the places in $S$. By definition, places outside of $S$ are finite and odd. Locally at $S_{n_1}$ and $S_{n_2}$, the cup product $\chi(1) \cup \chi(3)$ is trivial. At the remaining places, both $\chi(1)$ and $\chi(3)$ are unramified, which implies that the cup product $\chi(1) \cup \chi(3)$ is locally trivial there as well.
\end{proof}

\end{document}